\documentclass[
	final,
	paper=a4,
	pagesize=auto,
	fontsize=10pt,
	version=last,
]{scrartcl}

\usepackage[utf8]{inputenc}
\usepackage[
	english,
]{babel}

\KOMAoptions{%
   DIV=10,
}%
\KOMAoptions{%
   numbers=enddot
}%
\KOMAoptions{
   twoside=semi,  
   twocolumn=false, 
   headinclude=false,%
   footinclude=false,%
   mpinclude=false,%
   headlines=2.1,%
   headsepline=true,%
   footsepline=false,%
   cleardoublepage=empty 
}%
\KOMAoptions{%
   parskip=half-  
}%
\KOMAoptions{%
}%
\KOMAoptions{%
}%
\KOMAoptions{%
   bibliography=openstyle,%
}%
\KOMAoptions{%
}%
\KOMAoptions{%
}%
\KOMAoptions{%
}%

\makeatletter
\newcommand{\LoadPackagesNow}{}
\newcommand{\LoadPackageLater}[2][]{%
   \g@addto@macro{\LoadPackagesNow}{%
      \usepackage[#1]{#2}%
   }%
}
\makeatother

\usepackage{xspace}
\usepackage{xargs}
\usepackage{ifthen}
\usepackage{etoolbox}

\usepackage[
	margin=3cm,
	]{geometry}
\usepackage{framed}
\usepackage{mdframed}
\usepackage{pbox}
\usepackage{enumitem}
\usepackage{array}
\usepackage{multirow}
\usepackage{longtable}
\usepackage{hhline}
\usepackage{afterpage}
\usepackage{pdflscape}
\usepackage{rotating}
\usepackage{booktabs}
\usepackage{setspace}
\usepackage[
   bottom,      
   stable,      
   perpage,     
   ragged,      
   multiple,    
   flushmargin,
   hang,
]{footmisc}

\usepackage{mathpazo}

\usepackage[%
   headsepline,
   automark,
   komastyle,
]{scrpage2}

\clearscrheadings
\clearscrplain
\lehead{\pagemark}
\rohead{\pagemark}
\rehead{\headmark}
\lohead{\shortauthor}
\automark[section]{section} 

\usepackage{relsize}
\usepackage[normalem]{ulem}      
\usepackage{soul}		           
\usepackage{url}
\usepackage{lipsum}

\usepackage[table]{xcolor}
\usepackage[%
]{graphicx}

\usepackage{epstopdf}
\usepackage{wrapfig}
\usepackage{caption}
\usepackage{subcaption}
\usepackage[
   tbtags,    
   sumlimits,  
   nointlimits, 
   namelimits, 
   reqno,     
]{amsmath} %

\usepackage{amsfonts}
\usepackage{mathrsfs} 
\usepackage{dsfont}
\usepackage{amssymb}
\usepackage{units}
\LoadPackageLater{amsthm}
\usepackage[fixamsmath,disallowspaces]{mathtools}
\mathtoolsset{showonlyrefs}
\mathtoolsset{centercolon=true}
\usepackage{bm} 
\makeatletter
\g@addto@macro\bfseries{\boldmath}
\makeatother
\allowdisplaybreaks[4]
\numberwithin{equation}{section}

\usepackage[toc]{appendix}

\usepackage[%
	square,	
	comma,	
	numbers,	
	sort,		
	sort&compress,    
]{natbib}

\usepackage{csquotes}

\usepackage[textsize=tiny,english,colorinlistoftodos,
	disable,
	]{todonotes}

\definecolor{pdfurlcolor}{rgb}{0,0,0.6}
\definecolor{pdffilecolor}{rgb}{0.7,0,0}
\definecolor{pdflinkcolor}{rgb}{0,0,0.6}
\definecolor{pdfcitecolor}{rgb}{0,0,0.6}
\usepackage[
  colorlinks=true,         
  urlcolor=pdfurlcolor,    
  filecolor=pdffilecolor,  
  linkcolor=pdflinkcolor,  
  citecolor=pdfcitecolor,  %
  raiselinks=true,			 
  breaklinks,              
  verbose,
  hyperindex=true,         
  linktocpage=true,        
  hyperfootnotes=false,     
  bookmarks=true,          
  bookmarksopenlevel=1,    
  bookmarksopen=true,      
  bookmarksnumbered=true,  
  bookmarkstype=toc,       
  plainpages=false,        
  pageanchor=true,         
  pdfdisplaydoctitle=true, 
  pdfstartview=FitH,       
  pdfpagemode=UseOutlines, 
  pdfpagelabels=true,           
  pdfpagelayout=OneColumn, 
]{hyperref}

\LoadPackagesNow

\newcommand{\ifargdef}[3][{}]{\ifthenelse{\equal{#2}{}}{#1}{#3}}

\setkomafont{section}{\centering\normalfont\rmfamily\scshape}
\setkomafont{subsection}{\rmfamily\bfseries}
\setkomafont{paragraph}{\rmfamily\bfseries}
\setkomafont{pageheadfoot}{\smaller\scshape\rmfamily}

\hypersetup{
	pdfauthor={Martin Genzel},
	pdftitle={High-Dimensional Estimation of Structured Signals from Non-Linear Observations with General Convex Loss Functions},
	pdfsubject={Article},
	pdfcreator={PDF-LaTeX},
}

\newenvironment{highlight}{\begin{addmargin}[1em]{1em}\itshape}{\end{addmargin}}

\newenvironment{properties}[2][2em]
{\begin{enumerate}[label={\textsc{(#2\arabic*)}},leftmargin=#1]}
{\end{enumerate}} 

\newenvironment{listing}
{\begin{itemize}[itemindent=0em,leftmargin=1.2em]}
{\end{itemize}}

\newenvironment{rmklist}
{\begin{enumerate}[label={(\arabic*)},itemindent=2em,leftmargin=0em]}
{\end{enumerate}}

\newenvironment{thmlist}
{\begin{enumerate}[label={(\arabic*)}]}
{\end{enumerate}}

\newenvironment{proofsteps}
{\begin{enumerate}[label={(\arabic*)},itemindent=2em,leftmargin=0em]}
{\end{enumerate}}

\renewcommand{\qedsymbol}{$_\blacksquare$}
\newcommand{\qeddiamond}{\hfill$\Diamond$}
\providecommand{\qedhere}{\hfill\qedsymbol}

\newtheoremstyle{claim}
	{\topsep}{\topsep}%
	{\itshape}
	{}
	{}
	{}
	{.5em}
	{{\bfseries\boldmath\thmname{#1} \thmnumber{#2}} \thmnote{(#3)}}

\newtheoremstyle{definition}
	{\topsep}{\topsep}%
	{}
	{}
	{}
	{}
	{.5em}
	{\textbf{\thmname{#1} \thmnumber{#2}} \thmnote{(#3)}}
	
\newtheoremstyle{algorithm}
	{\topsep}{\topsep}%
	{}
	{}
	{\bfseries\boldmath}
	{}
	{.5em}
	{\thmname{#1} \thmnumber{#2} \thmnote{(#3)}}

\theoremstyle{claim}
\newtheorem{theorem}{Theorem}[section]
\newtheorem{lemma}[theorem]{Lemma}
\newtheorem{corollary}[theorem]{Corollary}
\newtheorem{proposition}[theorem]{Proposition}
\theoremstyle{definition}
\newtheorem{definition}[theorem]{Definition}
\newtheorem{remark}[theorem]{Remark}
\newtheorem{example}[theorem]{Example}
\newcommand{\opleft}[1]{\mathopen{}\left#1}
\newcommand{\opright}[1]{\right#1\mathclose{}}
\newcommandx{\braces}[4]{%
\ifstrequal{#3}{normal}{#1#4#2}{%
\ifstrequal{#3}{auto}{\left#1#4\right#2}{%
\ifstrequal{#3}{opauto}{\opleft#1#4\opright#2}{%
#3#1#4#3#2}}}%
}
\newcommandx{\opannot}[3][3=\downarrow]{\stackrel{\mathclap{\substack{#1 \\ #3 \vspace{2pt}}}}{#2}}
\newcommandx{\lineannot}[3][3=\rightarrow]{\mathllap{\boxed{\text{\textsmaller{#1}}} #3} #2}
\newcommandx{\multilineannot}[4][4=\rightarrow]{\mathllap{\boxed{\parbox{#1}{\RaggedRight\textsmaller{#2}}} #4} #3}
 
\newcommand{\N}{\mathbb{N}} 
\newcommand{\R}{\mathbb{R}} 
\newcommand{\eps}{\varepsilon} 

\newcommand{\suchthat}[1][normal]{\ifstrequal{#1}{normal}{\mid}{#1|}} 

\newcommand{\cardinality}{\#} 
\newcommand{\union}{\cup} 
\newcommand{\intersec}{\cap} 

\newcommandx{\intvcl}[3][1=normal]{\braces{[}{]}{#1}{#2, #3}} 
\newcommandx{\intvop}[3][1=normal]{\braces{(}{)}{#1}{#2, #3}} 
\newcommandx{\intvclop}[3][1=normal]{\braces{[}{)}{#1}{#2, #3}} 
\newcommandx{\intvopcl}[3][1=normal]{\braces{(}{]}{#1}{#2, #3}} 
\DeclareMathOperator*{\argmin}{argmin} 
\DeclareMathOperator{\sign}{sign}
\newcommandx{\abs}[2][1=normal]{\braces{\lvert}{\rvert}{#1}{#2}} 
\newcommandx{\ceil}[2][1=normal]{\braces{\lceil}{\rceil}{#1}{#2}} 
\newcommandx{\floor}[2][1=normal]{\braces{\lfloor}{\rfloor}{#1}{#2}} 
\newcommandx{\round}[2][1=normal]{\braces{[}{]}{#1}{#2}} 
\newcommandx{\der}[1]{D^{#1}} 
\newcommandx{\gradient}{\nabla} 
\newcommandx{\partder}[4][1={},4={}]{\frac{\partial^{#4} #2}{\partial #3^{#4}}\ifargdef{#1}{\Big|_{#1}}} 
\newcommandx{\integ}[4][1={},2={}]{\int_{#1}^{#2} #3 \, #4} 
\newcommandx{\asympffaster}[2][1=normal]{o\braces{(}{)}{#1}{#2}} 
\newcommandx{\asympfaster}[2][1=normal]{O\braces{(}{)}{#1}{#2}} 
\newcommandx{\asympeq}[2][1=normal]{\Theta\braces{(}{)}{#1}{#2}} 
\newcommandx{\asympsslower}[2][1=normal]{\omega\braces{(}{)}{#1}{#2}} 
\newcommandx{\asympslower}[2][1=normal]{\Omega\braces{(}{)}{#1}{#2}} 

\newcommand{\matr}[1]{\begin{bmatrix} #1 \end{bmatrix}} 
\newcommandx{\norm}[2][1=normal]{\braces{\|}{\|}{#1}{#2}} 
\renewcommandx{\sp}[3][1=normal]{\braces{\langle}{\rangle}{#1}{#2, #3}} 
\newcommandx{\End}[2][2={}]{\mathcal{L}\opleft( #1 \ifargdef{#2}{, #2} \opright)} 
\newcommand{\orthcompl}[1]{{#1}^\perp} 
\newcommand{\T}{\mathsf{T}} 
\renewcommand{\vec}[1]{\boldsymbol{#1}} 

\newcommandx{\measure}[2][1=normal]{\operatorname{vol}\braces{(}{)}{#1}{#2}} 
\newcommand{\indset}[1]{\chi_{#1}} 
\DeclareMathOperator{\supp}{supp} 
\newcommandx{\Leb}[3][1={},3=normal]{L^{#2}\ifargdef{#1}{\braces{(}{)}{#3}{#1}}{}} 
\newcommandx{\Lebnorm}[4][1=normal,3={2},4={}]{\norm[#1]{#2}_{#3}} 
\renewcommandx{\l}[3][1={},3=normal]{\ell^{#2}\ifargdef{#1}{\braces{(}{)}{#3}{#1}}} 
\newcommandx{\lnorm}[4][1=normal,3={2},4={}]{\norm[#1]{#2}_{#3}} 
\newcommandx{\Smooth}[4][1={},3={},4=normal]{C_{#3}^{#2}\ifargdef{#1}{\braces{(}{)}{#4}{#1}}} 
\newcommandx{\Schwartz}[2][1={},2=normal]{\mathscr{S}\ifargdef{#1}{\braces{(}{)}{#2}{#1}}} 
\newcommandx{\Schwartzpoly}[2][1=normal]{\braces{\langle}{\rangle}{#1}{\abs[#1]{#2}} } 
\newcommandx{\Tempdistr}[2][1={},2=normal]{\mathscr{S}'\ifargdef{#1}{\braces{(}{)}{#2}{#1}}} 
\newcommandx{\distrinp}[3][1=normal]{\braces{\langle}{\rangle}{#1}{#2, #3}} 
\newcommandx{\ft}[3][1=default,2=auto]{
\ifstrequal{#1}{default}{\widehat{#3}}{
\ifstrequal{#1}{long}{{\braces{(}{)}{#2}{#3}}^{\wedge}}{}}} 
\newcommandx{\ift}[3][1=default,2=auto]{
\ifstrequal{#1}{default}{\check{#3}}{
\ifstrequal{#1}{long}{{\braces{(}{)}{#2}{#3}}^{\vee}}{}}} 

\newcommand{\define}[1]{\emph{#1}}

\newcommand{\y}{\vec{y}}
\newcommand{\yadv}{\vec{\tilde{y}}}
\newcommand{\Y}{\mathcal{Y}}
\newcommand{\A}{\vec{A}}
\newcommand{\Aind}{\vec{B}}
\renewcommand{\a}{\vec{a}}
\newcommand{\z}{\vec{z}}
\newcommand{\zadv}{\vec{\tilde{z}}}
\newcommand{\x}{\vec{x}}
\newcommand{\grtr}{\vec{x}_0}
\newcommand{\grtrmu}{\scalfac\vec{x}_0}
\newcommand{\solu}{\vec{\hat{x}}}
\newcommand{\sset}{K}
\newcommand{\h}{\vec{h}}
\newcommand{\dict}{\vec{D}}
\newcommand{\proj}{\vec{P}}

\newcommand{\advfrac}{\tau}
\newcommand{\advdev}{\eps}
\newcommand{\loss}{\mathcal{L}}
\newcommand{\losssq}{\mathcal{L}^{\text{sq}}}
\newcommand{\score}{\mathcal{S}}
\newcommand{\lossemp}[1][{}]{\bar{\mathcal{L}}_{#1}}
\newcommand{\losssqemp}[1][{}]{\bar{\mathcal{L}}_{#1}^{\text{sq}}}
\newcommand{\losstaylor}[3]{\delta\lossemp[#3](#1, #2)}
\newcommand{\losssqtaylor}[3]{\delta\losssqemp[#3](#1, #2)}
\newcommand{\RSC}{{\text{RSC}}}
\newcommand{\fobs}{f}
\newcommand{\scalfac}{\mu}
\newcommand{\modeldev}{\sigma}
\newcommand{\modeldeveta}{\eta}
\newcommand{\modeldevconst}{C_{\modeldev,\modeldeveta}}
\newcommand{\accuracy}{t_0}
\newcommand{\I}{\vec{I}_n}
\newcommand{\normpol}[2]{\norm{#1}_{{#2}^\circ}}

\newcommandx{\prob}[2][1={},2=normal]{\mathbb{P}\ifargdef{#1}{\braces{[}{]}{#2}{#1}}}

\newcommandx{\mean}[2][1={},2=normal]{\mathbb{E}\ifargdef{#1}{\braces{[}{]}{#2}{#1}}}
\newcommandx{\var}[2][1={},2=normal]{\mathbb{V}\ifargdef{#1}{\braces{[}{]}{#2}{#1}}}
\newcommand{\indprob}[1]{\mathds{1}[#1]} 

\newcommand{\distributed}{\sim}

\newcommand{\Normdistr}[2]{\mathcal{N}(#1, #2)}
\newcommand{\gaussian}{\vec{g}}
\newcommand{\gaussianuniv}{g}

\newcommand{\Covmatr}{\vec{\Sigma}}

\newcommand{\empproc}{\mathsf{F}}

\newcommandx{\ball}[2][1={},2={}]{B_{#1}^{#2}}
\DeclareMathOperator{\convhull}{conv}
\renewcommand{\S}{S}

\newcommand{\meanwidth}[2][{}]{w_{#1}(#2)}
\newcommand{\effdim}[2][{}]{d_{#1}(#2)}
\newcommand{\cone}[2]{\mathcal{C}(#1,#2)}

\graphicspath{{images/}}

\begin{document}

\pagestyle{scrheadings}

\noindent{\Large\raggedright\bfseries High-Dimensional Estimation of Structured Signals from \\ Non-Linear Observations with General Convex Loss Functions}

\vspace{1\baselineskip}
\begin{addmargin}[6em]{0em}
\noindent{\normalsize\bfseries\larger{Martin Genzel}}

\noindent Technische Universit\"at Berlin, Department of Mathematics \\
Straße des 17. Juni 136, 10623 Berlin, Germany

\noindent E-Mail: \href{mailto:genzel@math.tu-berlin.de}{\texttt{genzel@math.tu-berlin.de}}

\vspace{1\baselineskip}
{\smaller
\noindent\textbf{Abstract.} 
In this paper, we study the issue of estimating a \emph{structured signal} $\grtr \in \R^n$ from \emph{non-linear  and noisy Gaussian observations}.
Supposing that $\grtr$ is contained in a certain \emph{convex subset} $\sset \subset \R^n$, we prove that accurate recovery is already feasible if the number of observations exceeds the \emph{effective dimension} of $\sset$, which is a common measure for the complexity of signal classes.
It will turn out that the possibly unknown non-linearity of our model affects the error rate only by a multiplicative constant.
This achievement is based on recent works by Plan and Vershynin, who have suggested to treat the non-linearity rather as \emph{noise} which perturbs a linear measurement process.
Using the concept of \emph{restricted strong convexity}, we show that their results for the \emph{generalized Lasso} can be extended to a fairly large class of \emph{convex loss functions}.
Moreover, we shall allow for the presence of \emph{adversarial noise} so that even deterministic model inaccuracies can be coped with.
These generalizations particularly give further evidence of why many standard estimators perform surprisingly well in practice, although they do not rely on any knowledge of the underlying output rule.
To this end, our results provide a unified and general framework for signal reconstruction in high dimensions, covering various challenges from the fields of \emph{compressed sensing}, \emph{signal processing}, and \emph{statistical learning}.

}
\end{addmargin}
\newcommand{\shortauthor}{M. Genzel}

\thispagestyle{plain}

\section{Introduction}

\subsection{Motivation}

Before we introduce the general setup of non-linear measurements and structured signals, let us first consider the classical problem of (high-dimensional) estimation from a \define{linear model}. In this situation, the goal is to recover an unknown \define{signal} $\grtr \in \R^n$ from a set of noisy linear \define{observations} (or \define{measurements})
\begin{equation}\label{eq:intro:linmodel}
y_i = \sp{\a_i}{\grtr} + z_i, \quad i = 1, \dots, m,
\end{equation}
where $\a_1, \dots, \a_m \in \R^n$ are known vectors, defining a \define{measurement process}, and $z_1, \dots, z_m \in \R$ are small (random) perturbations. Obviously, a unique reconstruction of $\grtr$ is impossible in general if $m < n$, even if $z_1 = \dots = z_m = 0$. However, when we assume some additional \define{structure} for our signal-of-interest, such as \define{sparsity},\footnote{By \define{sparsity}, we mean that only a very few entries of $\grtr$ are non-zero.} the problem often becomes feasible and numerous efficient algorithms are available for recovery. One of the most popular approaches is the \define{Lasso} (Least Absolute Shrinkage and Selection Operator), which was originally proposed by Tibshirani \cite{tibshirani1996lasso}:
\begin{equation}
	\min_{\x \in \R^n} \tfrac{1}{2m}\sum_{i = 1}^m (y_i - \sp{\a_i}{\x})^2 \quad \text{subject to $\lnorm{\x}[1] \leq R$.} \label{eq:intro:lasso}\tag{$P_{R\ball[1][n]}$}
\end{equation}
The purpose of this convex program is to perform a least-squares fit to the observations $y_1, \dots, y_m$, whereas the $\l{1}$-constraint encourages a certain degree of sparsity of the solution, controlled by the parameter $R > 0$.

A fundamental question is now the following: 
\begin{highlight}
	How many measurements $m$ are required to achieve an accurate recovery of $\grtr$ by \eqref{eq:intro:lasso}?
\end{highlight}
Interestingly, it has turned out that only about $s\log(\frac{2n}{s})$ measurements are needed to reconstruct an $s$-sparse vector.\footnote{A vector $\grtr \in \R^n$ is called \define{$s$-sparse} if at most $s$ of its entries are non-zero. More precisely, the error of the reconstruction can be bounded in terms of the present noise level.} This (optimal) rate can be realized by randomizing the measurement process, for instance, by choosing i.i.d. Gaussian vectors $\a_i \distributed \Normdistr{\vec{0}}{\I}$. Perhaps the most crucial observation is that $s\log(\frac{2n}{s})$ only \emph{logarithmically} depends on the dimension of the ambient space $\R^n$. This indicates that good estimates are even possible in a high-dimensional setting where $m \ll n$, supposed that the signal-of-interest exhibits some low-dimensional structure. Results of this type are usually associated with the field of \define{compressed sensing}, which originally emerged from the works of Cand\`{e}s, Donoho, Romberg, and Tao \cite{donoho2006cs,candes2005decoding,candes2006recovery,candes2006stable}.

However, the assumptions of a linear model and exact sparsity are often too restrictive for real-world applications. For this reason, there has been a remarkable effort during the previous decade to extend the above setup and corresponding guarantees in several directions.
As an example, one could think of non-linear \define{binary observations} where only the sign of the linear measurement is retained:
\begin{equation}\label{eq:intro:onebitmodel}
y_i = \sign(\sp{\a_i}{\grtr} + z_i), \quad i = 1, \dots, m.
\end{equation}
This situation is closely related to \define{$1$-bit compressed sensing} \cite{boufounos2008onebit} and \define{classification problems}. Apart from that, it is also essential to allow for more flexible signal structures, e.g., \define{approximately sparse vectors}, \define{sparse dictionary representations}  \cite{candes2011csdict,rauhut2008csdict}, or \define{model-based compressed sensing} \cite{baraniuk2010modelcs}. In such a general framework, one might ask again whether the original Lasso-estimator \eqref{eq:intro:lasso} is still able to produce good results, or whether it might be more beneficial to replace the square loss by a specifically adapted functional. We shall address all of these questions later on, including the impact of non-linear measurements, arbitrary convex signal sets, and general convex loss functions. In particular, it will be demonstrated how these different ingredients ``interact'' with each other and how the actual recovery performance is affected by them.

\subsection{Observation Model and Gaussian Mean Width}
\label{subsec:intro:model}

As a first step, let us replace the noisy linear model of \eqref{eq:intro:linmodel} by a more general observation rule. Throughout this paper, we will consider a \define{semiparametric single-index model}, which follows the approach of \cite{plan2014highdim,plan2015lasso}:
\begin{properties}[3em]{M}
\item\label{assump:model:nonlinmod}
	We assume that the \define{observations} obey 
	\begin{equation}\label{eq:intro:model}
	y_i := \fobs(\sp{\a_i}{\grtr}), \quad i = 1, \dots, m,
	\end{equation}
	where $\grtr \in \R^n$ is the \define{ground-truth signal}, $\a_i \distributed \Normdistr{\vec{0}}{\Covmatr}$ are \emph{independent} mean-zero Gaussian vectors with positive definite covariance matrix $\Covmatr \in \R^{n \times n}$, and $\fobs \colon \R \to \Y$ is a (possibly random) function which is independent of $\a_i$.\footnote{The randomness of $\fobs$ is understood measurement-wise, i.e., for every observation $i \in \{1, \dots, m\}$, we have an \emph{independent} copy of $\fobs$. But note that this explicit dependence of $\fobs$ on $i$ is omitted.} 
	Here, $\Y$ is always assumed to be a \emph{closed} subset of $\R$.
	For the sake of compact notation, we also introduce $\y := (y_1, \dots, y_m) \in \Y^m$ and $\A := \matr{\a_1 & \dots & \a_m}^\T \in \R^{m \times n}$.
\end{properties}
The range $\Y$ of $\fobs$ might restrict the set of possible observations. For example, one would have $\Y = \{-1,0,+1\}$ if $\fobs(v) = \sign(v)$.
In general, the function $\fobs$ plays the role of a \define{non-linearity}\footnote{This is somewhat an abuse of notation: Although we always speak of a ``non-linearity,'' this also covers the case of $\fobs$ being the identity (plus noise).} which modifies the output of the linear projection $\sp{\a_i}{\grtr}$ in a certain way. This particularly includes perturbations like additive random noise or random bit-flips (when dealing with binary measurements such as in \eqref{eq:intro:onebitmodel}). Remarkably, it will turn out that our recovery method for $\grtr$ does not explicitly rely on $\fobs$ so that the non-linearity of the model---and especially the distribution of noise---can be even \emph{unknown}.

However, in realistic situations, it might be still too restrictive to assume that the distortion of the (linear) model is \emph{independent} from the measurement process $\A$. We shall therefore also allow for \define{adversarial noise}:
\begin{properties}[3em]{M}
\setcounter{enumi}{1}
\item\label{assump:model:advnoise}
	We assume that the actual observations are given by a vector $\yadv = (\tilde{y}_1, \dots, \tilde{y}_m) \in \Y^m$, which could differ from $\y$. 
	The deviation between $\yadv$ and $\y$ is measured by means of the following \define{adversarial noise parameter}:
	\begin{equation}\label{eq:intro:advnoise}
		\advdev := \Big(\tfrac{1}{m} \sum_{i = 1}^m \abs{\tilde{y}_i - y_i}^2\Big)^{1/2} \geq 0.
	\end{equation}	
\end{properties}
Note that the mismatch between $\tilde{y}_i$ and $y_i$ could be even deterministic or depend on $\a_i$.
We will observe that the error bounds of our main results involve an ``inevitable'' additive term of $\advdev$, which quantifies the level of adversarial noise.

In order to impose a certain structural constraint on the ground-truth signal, we simply assume that $\grtr$ is contained in a \define{signal set} $\sset \subset \R^n$. For example, $\sset$ could be the set of all $s$-sparse vectors or an appropriate convex relaxation of it. Now, we can make our major goal precise:
\begin{highlight}
	Given an observation vector $\yadv \in \Y^m$ and the underlying measurement matrix $\A \in \R^{m \times n}$, find an estimator $\solu \in \sset$ which is close to $\grtr$ (measured in the Euclidean distance\footnote{In this paper, we will measure the estimation error only in terms of the $\l{2}$-norm, but there are of course many other ways to do this.}).
\end{highlight}
The actual ability to recover $\grtr$ will particularly depend on the ``complexity'' of the signal set $\sset$. For instance, it is not very surprising that a signal can be estimated more easily if it is known to belong to a low-dimensional subspace, instead of the entire $\R^n$. A very common and powerful measure of complexity is the so-called \define{(Gaussian) mean width}:
\begin{definition}\label{def:intro:meanwidth}
	The \define{(global Gaussian) mean width} of a set $L \subset \R^n$ is given by
	\begin{equation}\label{eq:intro:meanwidth}
		\meanwidth{L} := \mean[\sup_{\x \in L} \sp{\gaussian}{\x}],
	\end{equation}
	where $\gaussian \distributed \Normdistr{\vec{0}}{\I}$ is a standard Gaussian random vector.
\end{definition}
\begin{figure}
	\centering
	\includegraphics[width=.5\textwidth]{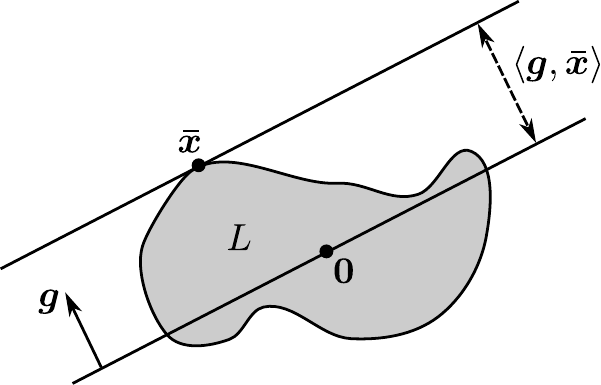}
	\caption{Geometric meaning of the mean width. If $\lnorm{\gaussian} = 1$, then $\sp{\gaussian}{\vec{\bar{x}}} = \sup_{\x \in L} \sp{\gaussian}{\x}$ measures the spatial extent of $L$ in direction of $\gaussian$, that is, the distance between the hyperplanes $\orthcompl{\{\gaussian\}}$ and $\orthcompl{\{\gaussian\}} + \vec{\bar{x}}$. Thus, computing the expected value in \protect\eqref{eq:intro:meanwidth} yields an ``average'' measure for the size of $L$.}
	\label{fig:intro:meanwidth}
\end{figure}
The geometric meaning of Definition~\ref{def:intro:meanwidth} is illustrated in Figure~\ref{fig:intro:meanwidth}.
The notion of mean width originates from the fields of \define{geometric functional analysis} and \define{convex geometry}. But in equivalent forms, it also appears as \define{$\gamma_2$-functional} in stochastic processes (cf. \cite{talagrand2014chaining}) or as \define{Gaussian complexity} in learning theory (cf. \cite{bartlett2003complexity}).
For an intuitive understanding of the mean width, it is helpful to regard its square $\effdim{L} := \meanwidth{L}^2$ as the \define{effective dimension} of the set $L$. For example, if $L = \{ \x \in \R^n \suchthat \text{$\x$ is $s$-sparse and $\lnorm{\x} \leq 1$} \}$, we have $\effdim{L} \asymp s \log(2n/s) \ll n$ (see \cite[Lem.~2.3]{plan2013robust}).\footnote{Here, $\effdim{L} \asymp s \log(2n/s)$ means that there exist constants $C_1, C_2 > 0$ such that $C_1 s \log(2n/s) \leq \effdim{L} \leq C_2 s \log(2n/s)$.} This observation reflects the low-complexity nature of $s$-sparse vectors and it particularly indicates that the effective dimension of a set $L$ can be relatively small even though having full algebraic dimension.
In general, the mean width is robust against small perturbations, implying that slightly increasing $L$ will only slightly change $\meanwidth{L}$ (cf. \cite{plan2014highdim}).

Regarding signal estimation, we will conclude later on that the feasibility of recovery is actually established by this single quantity only: Very roughly stated, an accurate approximation of $\grtr$ is possible as long as the number of available measurements $m$ exceeds the effective dimension of our signal set $\sset$.

\subsection{Generalized Estimator}
\label{subsec:intro:estimator}

Next, we would like generalize the classical Lasso-estimator \eqref{eq:intro:lasso} in such a way that it becomes capable of fitting non-linear models in arbitrary signal sets $\sset$. Perhaps the most straightforward strategy here is to adapt the $\l{1}$-constraint, i.e.,
\begin{equation}
	\min_{\x \in \R^n} \tfrac{1}{2m}\sum_{i = 1}^m (\tilde{y}_i - \sp{\a_i}{\x})^2 \quad \text{subject to $\x \in \sset$.} \label{eq:intro:klasso}\tag{$P_{\sset}$}
\end{equation}
This optimization program is known as the \define{$\sset$-Lasso} or \define{generalized Lasso}. Plan and Vershynin have shown in \cite{plan2015lasso} that, under the model assumption of \ref{assump:model:nonlinmod}, a successful recovery of $\grtr$ by \eqref{eq:intro:klasso} can indeed be guaranteed, at least with high probability. Such an achievement might be a bit astonishing at first sight because one actually tries to ``fool'' the above $\sset$-Lasso by fitting a \emph{linear} rule $y_i^{\text{lin}} := \sp{\a_i}{\x}$ to \emph{non-linear} observations $\tilde{y}_i$.
The main idea behind these results is that---inspired by the noisy linear model of \eqref{eq:intro:linmodel}---the non-linearity is rather treated as \emph{noise} which disturbs a linear measurement process. 
Therefore, the following three parameters play a key role in \cite{plan2015lasso}:
\noeqref{eq:intro:modelparam:modeldev} \noeqref{eq:intro:modelparam:modeldeveta}%
\begin{subequations} \label{eq:intro:modelparam}
\begin{align} 
\scalfac &:= \mean[\fobs(\gaussianuniv) \cdot \gaussianuniv], \label{eq:intro:modelparam:scalfac} \\
\modeldev^2 &:= \mean[(\fobs(\gaussianuniv) - \scalfac\gaussianuniv)^2], \label{eq:intro:modelparam:modeldev} \\
\modeldeveta^2 &:= \mean[(\fobs(\gaussianuniv) - \scalfac\gaussianuniv)^2\cdot \gaussianuniv^2], \label{eq:intro:modelparam:modeldeveta}
\end{align}
\end{subequations}
where $\gaussianuniv \distributed \Normdistr{0}{1}$. Here, $\scalfac$ can be regarded as the \emph{correlation} between the linear and non-linear model, while $\modeldev^2$ and $\modeldeveta^2$ basically capture the variance between them. This intuition is also underpinned by the noisy linear case $y_i = f(\sp{\a_i}{\grtr}) := \tilde{\scalfac} \sp{\a_i}{\grtr} + \tilde{\modeldev} z_i$ with $z_i \distributed\Normdistr{0}{1}$ in which we simply obtain $\scalfac = \tilde{\scalfac}$ and $\modeldev^2 = \modeldeveta^2 = \tilde{\modeldev}^2$ from \eqref{eq:intro:modelparam}.

In practice however, the Lasso is often outperformed by approaches which are adapted to specific problem situations. For example, it would be more natural to apply \define{logistic regression} when dealing with binary measurements such as in \eqref{eq:intro:onebitmodel}. As a consequence, we shall replace the square loss of \eqref{eq:intro:klasso} by a general \define{loss function} 
\begin{equation}
\loss \colon \R \times \Y \to \R, \ (v,y) \mapsto \loss(v, y),
\end{equation}
allowing us to measure the \define{residual} between $v = \sp{\a_i}{\x}$ and $y = \tilde{y}_i$ in a different way.
This leads to the following \define{generalized estimator}:
\begin{equation}
	\min_{\x \in \R^n} \tfrac{1}{m} \sum_{i = 1}^m \loss(\sp{\a_i}{\x}, \tilde{y}_i) \quad \text{subject to $\x \in \sset$.} \label{eq:intro:estimator}\tag{$P_{\loss, \sset}$}
\end{equation}
Note that, when choosing the square loss $\losssq(v, y) := \tfrac{1}{2} (v-y)^2$, $v,y\in \R$ in \eqref{eq:intro:estimator}, we precisely end up with the $\sset$-Lasso \eqref{eq:intro:klasso}.\footnote{The reason for using a factor of $\tfrac{1}{2}$ is just for the sake of convenience, since we shall frequently work with the first derivative of $\loss$.} 
Sometimes, it will be beneficial to express the objective functional of \eqref{eq:intro:estimator} as a mapping of the signal:
\begin{equation}
	\lossemp[\yadv]\colon \R^n \to \R, \ \x \mapsto \lossemp[\yadv](\x) := \tfrac{1}{m} \sum_{i = 1}^m \loss(\sp{\a_i}{\x}, \tilde{y}_i).
\end{equation}
Usually, $\lossemp[\yadv]$ is referred to as the \define{empirical loss function}, since it tries to approximate the expected loss by computing an empirical mean value. But one should be aware of the fact that $\lossemp[\yadv]$ still depends on $\A$ and is therefore a random function. Nevertheless, this dependence will be omitted for the sake of readability.
\begin{remark}\label{rmk:intro:loss}
\begin{rmklist}
\item
	If $\sset$ and $\lossemp[\yadv]$ are both convex, the optimization problem $\eqref{eq:intro:estimator}$ becomes convex as well, and efficient solvers are often available in practice.
	We will not further discuss computational issues and the uniqueness of solutions here. For those aspects, the interested reader is referred to \cite{tibshirani2011lasso,tibshirani2013lasso,rosset2007piecewise,efron2004lars}.
	In fact, the proofs of our results reveal that the recovery guarantees hold for \emph{any} minimizer of \eqref{eq:intro:estimator}.
\item\label{rmk:intro:loss:riskmin}
	In the context of statistical learning, the generalized estimator \eqref{eq:intro:estimator} is usually associated with \define{structural risk minimization}. This basically means that we would like to minimize the \define{risk} (measured by means of $\loss(\sp{\a_i}{\x}, \tilde{y}_i)$) of ``wrongly'' predicting the output $\tilde{y}_i$ by $\sp{\a_i}{\x}$. But using this language seems to be a bit inappropriate in our setup, since we are actually estimating from \emph{non-linear} observations.\qeddiamond
\end{rmklist}
\end{remark}

Inspired by the successful analysis of the $\sset$-Lasso in \cite{plan2015lasso}, the following question arises:
\begin{highlight}
What general properties should a loss function $\loss$ satisfy so that the minimizer of \eqref{eq:intro:estimator} provides an accurate estimator of the ground-truth signal $\grtr$?
\end{highlight}
First, let us assume convexity and some mild regularity for $\loss$: 
\begin{properties}[3em]{L}
\item\label{assump:estimator:difflip}
	Let $\loss$ be continuously differentiable in the first variable; the derivative with respect to the first variable is then denoted by $\loss'(v,y) = \partder{\loss}{v}(v,y)$. Furthermore, assume that $\loss'$ is Lipschitz continuous in the second variable, i.e., there exists a constant $C_{\loss'} > 0$ such that
	\begin{equation}
		\abs{\loss'(v, y) - \loss'(v, \tilde{y})} \leq C_{\loss'} \abs{y - \tilde{y}} \quad \text{for all $v \in \R, \ y,\tilde{y} \in \Y$.}
	\end{equation}
\item\label{assump:estimator:convex}
	The loss function $\loss$ is convex in the first variable (which is equivalent to assume that $\loss'$ is non-decreasing).
\end{properties}
\begin{figure}
	\centering
	\begin{subfigure}[t]{0.45\textwidth}
		\centering
		\includegraphics[width=\textwidth]{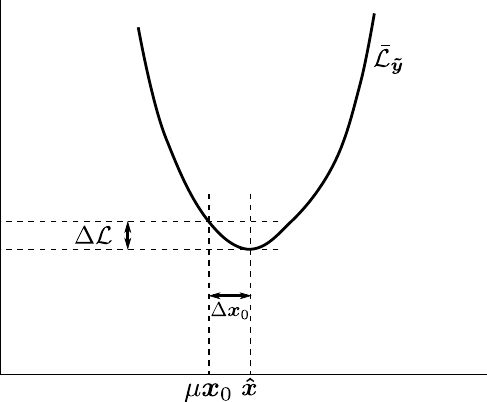}
		\caption{}
		\label{fig:intro:convexloss:nonflat}
	\end{subfigure}%
	\qquad
	\begin{subfigure}[t]{0.45\textwidth}
		\centering
		\includegraphics[width=\textwidth]{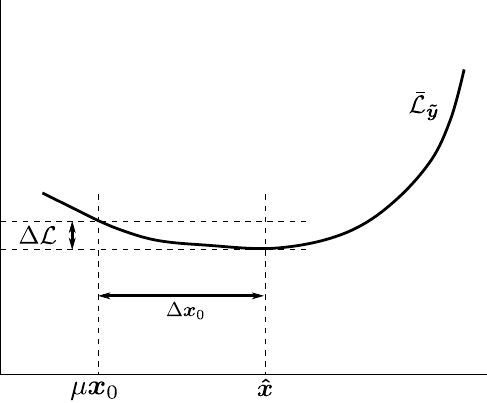}
		\caption{}
		\label{fig:intro:convexloss:flat}
	\end{subfigure}%
	\caption{\subref{fig:intro:convexloss:nonflat} $\lossemp[\yadv]$ is strongly curved so that small $\Delta\loss$ implies small $\Delta\grtr$. \subref{fig:intro:convexloss:flat} A less desirable situation where $\lossemp[\yadv]$ is relatively ``flat'' in the neighborhood of its minimizer.}
	\label{fig:intro:convexloss}
\end{figure}
As initial step towards successful signal estimation, one typically considers the empirical loss difference $\Delta\loss := \lossemp[\yadv](\grtrmu) - \lossemp[\yadv](\solu)$, where $\solu$ minimizes \eqref{eq:intro:estimator} and $\scalfac$ is a fixed scaling parameter defined similarly to \eqref{eq:intro:modelparam:scalfac}.\footnote{For now, we regard $\scalfac$ just as a fixed constant depending on the non-linear model \eqref{eq:intro:model} and the loss $\loss$. A precise definition will be given later in \eqref{eq:results:recovery:modelparam:scalfac}.} Under relatively mild conditions on $\loss$, one might show now that $\Delta\loss$ indeed converges to $0$ as $m \to \infty$.
This would imply the \define{(asymptotic) consistency} of the estimator \eqref{eq:intro:estimator} in the sense that it minimizes the risk of a wrong prediction (cf. Remark~\ref{rmk:intro:loss}\ref{rmk:intro:loss:riskmin}).
However, we are primarily interested in bounding the approximation error $\Delta\grtr := \lnorm{\solu - \grtrmu}$, which is of course a more difficult task. Our actual hope is that, supposed the empirical loss difference $\Delta\loss$ is small, the distance between $\solu$ and $\grtrmu$ is small as well.
From an analytical perspective, such a behavior depends on the \emph{curvature} of $\lossemp[\yadv]$; see also Figure~\ref{fig:intro:convexloss} for an illustration. 
In fact, it will turn out in the next subsection that \define{strict convexity} of $\loss$ (corresponding to strictly positive curvature) is already sufficient for establishing a bound for $\Delta\grtr$.

\subsection{A First Glimpse of Signal Recovery}

We are now ready to state a first recovery guarantee for $\grtr$. The following theorem, whose proof is given in Subsection~\ref{subsec:proof:introrecovery}, is an easy-to-read (though simplified) version of our main result presented in Section~\ref{sec:results}. 
Again, the model parameters $\scalfac$, $\modeldev$, and $\modeldeveta$ are adapted from \eqref{eq:intro:modelparam} and will be precisely defined in \eqref{eq:results:recovery:modelparam}.
\begin{theorem}\label{thm:intro:recovery}
Suppose that the assumptions \ref{assump:model:nonlinmod}, \ref{assump:model:advnoise}, \ref{assump:estimator:difflip}, \ref{assump:estimator:convex} hold with $\Covmatr = \I$. Let $\lnorm{\grtr} = 1$ and $\grtrmu \in \sset$ for a bounded, convex set $\sset \subset \R^n$. 
Moreover, let the loss function $\loss$ be twice continuously differentiable in the first variable with $\partder{\loss}{v}[2](v,y) \geq \mathcal{F}(v) > 0$ for all $(v,y) \in \R \times \Y$ and some continuous function $\mathcal{F}\colon \R \to \intvop{0}{\infty}$. Then, there exist a constant of the form $\modeldevconst = C \cdot \max\{1, \modeldev, \modeldeveta\} > 0$ such that the following holds with high probability:\footnote{The constant $C > 0$ may only depend on the ``probability of success'' as well as on the so-called RSC-constant of $\loss$ (see Definition~\ref{def:results:rsc}).} 
If the number of observations obeys
\begin{equation}\label{eq:intro:recovery:meascount}
m \geq C \cdot \meanwidth{\sset - \grtrmu}^2,
\end{equation}
then any minimizer $\solu$ of \eqref{eq:intro:estimator} satisfies
\begin{equation}\label{eq:intro:recovery:error}
	\lnorm{\solu - \grtrmu} \leq \modeldevconst \cdot \bigg( \Big(\frac{\meanwidth{\sset - \grtrmu}^2}{m}\Big)^{1/4} + \advdev \bigg).
\end{equation}
\end{theorem}
Roughly speaking, the error bound of \eqref{eq:intro:recovery:error}, combined with condition \eqref{eq:intro:recovery:meascount}, states the following:
\begin{highlight}
	A highly accurate estimate of $\grtr$ can be achieved if the number of observations (greatly) exceeds the effective dimension of $\sset - \grtrmu$ and the adversarial noise is not too strong, i.e., $\advdev$ is small.
\end{highlight}
Hence, Theorem~\ref{thm:intro:recovery} is particularly appealing for those signal sets of small mean width because only a relatively few measurements are required in this case. In Section~\ref{sec:consequences}, we will discuss this observation in greater detail and consider several examples of low-complexity signal sets. Perhaps the most remarkable feature of Theorem~\ref{thm:intro:recovery} is that the asymptotic error rate is essentially controlled by a single quantity, namely the mean width. 
In fact, we allow for a fairly general setting here, involving non-linear perturbations, adversarial noise, and strictly convex loss functions. The impact of the specific loss function and observation scheme is actually hidden by the constants $\scalfac$ and $\modeldevconst$, which in principle do not affect the capability of recovering the signal.
For practical purposes, the values of these constants are however relevant: Intuitively, $\modeldevconst$ should be small if the loss function and the \mbox{(non-)}linear model ``play well together,'' e.g., if the logistic loss is used to fit a logistic regression model. We shall make this heuristic more precise later on, when analyzing the (adapted) model parameters $\scalfac$, $\modeldev$, and $\modeldeveta$.


Although Theorem~\ref{thm:intro:recovery} already contains many important features of our recovery framework, several generalizations will be presented in Section~\ref{sec:results}, which also provide a better understanding of \emph{why} the proposed reconstruction method works. In particular, we shall replace the mean width by a \emph{localized} version, improving the above ``slow error rate'' of $O(m^{-1/4})$, and the convexity condition on $\loss$ is further relaxed by the notion of \define{restricted strong convexity}. Moreover, it will turn out that the assumption of isotropic measurement vectors ($\Covmatr = \I$) can be essentially dropped in Theorem~\ref{thm:intro:recovery}.

\subsection{Contributions and Related Work}

A major contribution of this paper is to develop a clear and unified framework for signal estimation in high dimensions, covering various problem scenarios from \emph{compressed sensing}, \emph{signal processing}, and \emph{statistical learning}.
One substantial conclusion from our results is that even standard estimators are surprisingly robust against model inaccuracies. Indeed, \eqref{eq:intro:estimator} is not based on any additional knowledge of the noisy (non-)linear distortion $\fobs$.
This particularly explains why it might be still promising to apply classical methods, such as the Lasso, although the true observation rule is largely unknown.
On the other hand, we need to ``pay the price'' that all of these uncertainties explicitly appear in the error bound.
Compared to many other approaches in machine learning, this is certainly a rather extreme perspective but it comes along with the great benefit that the actual estimator remains completely unaffected.

As already mentioned above, our main inspiration was the work of Plan and Vershynin in \cite{plan2015lasso}. Assuming the same output model \ref{assump:model:nonlinmod}, they have proven a similar reconstruction guarantee as in Theorem~\ref{thm:intro:recovery}, but only for the special case of the $\sset$-Lasso.
These results are primarily related to signal approximation and recovery, but the underlying proof techniques also rely on classical concepts from \define{geometric functional analysis} and \define{convex geometry}, such as Gaussian mean width or Gordon's escape.
For a brief discussion of these connections, the interested reader is referred to \cite[Sec.~II]{plan2015lasso} and \cite[Sec.~1.4]{plan2014highdim} as well as the references therein.
Furthermore, it should be emphasized that, in contrast to \cite{plan2015lasso}, our observation scheme additionally permits adversarial noise. This extension is especially relevant to real-world applications for which the assumption of \emph{independent} noise is too restrictive.

Another important achievement of this work is to go beyond the classical square loss towards general convex loss functions. In high-dimensional statistics, such a setup was first theoretically studied by the works of \cite{negahban2009unified,negahban2012unified}, which have isolated the crucial property of \define{restricted strong convexity} (\define{RSC}). The authors of \cite{negahban2009unified,negahban2012unified} focus on \define{$M$-estimators} with \define{decomposable regularizers}, instead of arbitrary convex signal sets. This setting is also quite general, and although our perspective is somewhat different here, we believe that similar results than ours could be deduced from their approach.

Inspired by \cite{negahban2009unified,negahban2012unified}, an adapted version of RSC is introduced in Section~\ref{sec:results} (see Definition~\ref{def:results:rsc}). In order to prove that our notion of RSC holds for a fairly large class of loss functions, we will apply \define{Mendelson's small ball method}, which is a powerful framework developed in a recent line of papers, including \cite{lecue2013learning,koltchinskii2015smallball,lecue2014cs,mendelson2014diameter,mendelson2014learning,mendelson2014learninggeneral}. A close connection to our results can be found in \cite{mendelson2014learning,mendelson2014learninggeneral}, where Mendelson presents a novel approach to learning and estimation problems that does not rely on traditional \emph{concentration inequalities}. 
Moreover, there is a recent work by Tropp \cite{tropp2014convex} which makes use of Mendelson's ideas to analyze convex programming methods for structured signal approximation.

Finally, we would like to mention another interesting line of research by Thrampoulidis, Oymak, Hassibi, and collaborators \cite{oymak2013lasso,oymak2013simple,thrampoulidis2014simple,thrampoulidis2015lasso}, providing \emph{exact} error bounds for the generalized Lasso with non-linear observations but rather in an asymptotic regime.

\begin{remark}
Although we mainly focus on the perspective of \emph{signal recovery} and \emph{compressed sensing} in this work, our problem setup is immediately related to issues from \emph{machine learning} and \emph{econometrics}. In fact, the model assumption of \ref{assump:model:nonlinmod} (with appropriate choices of $\fobs$) can be easily translated into (linear) \emph{regression}, \emph{classification}, \emph{generalized linear models}, etc. In this context, the vectors $\a_1, \dots, \a_m$ typically contain ``data'' describing a certain collection of \emph{feature variables}. Roughly speaking, our goal is then to \emph{learn} a \emph{parameter vector} $\grtr$ which allows us to predict the \emph{output variables} $y_1, \dots, y_m$ by means of \eqref{eq:intro:model}. A nice discussion about this correspondence can be found in \cite[Sec.~6]{plan2014highdim}. \qeddiamond
\end{remark}

\subsection{Outline and Notation}

In Section~\ref{sec:results}, we shall develop the ideas of the introduction further. For this purpose, the notions of \define{local mean width} and \define{restricted strong convexity} (\define{RSC}) are introduced, giving rise to our main recovery guarantee (Theorem~\ref{thm:results:recovery}). The next main result (Theorem~\ref{thm:results:rsc}) is then presented in Subsection~\ref{subsec:results:rsc}, stating a general sufficient condition for RSC (including the case of strictly convex loss functions). Finally, in Subsection~\ref{subsec:results:aniso}, we will show how the assumption of isotropic measurements can be dropped.
Note that all technical proofs are postponed to Section~\ref{sec:proof}, which also contains some remarks on the applied techniques.
Section~\ref{sec:consequences} is then devoted to several consequences and examples of our theoretical outcomes. 
This particularly includes common choices of signal sets, non-linearities, and loss functions. Apart from that, we will briefly discuss the underlying convex geometry of our problem setting in Subsection~\ref{subsec:consequences:convexgeo}.
Finally, some possible extensions of our framework are pointed out in Section~\ref{sec:conclusion}, which could be considered in future work.

Throughout this paper, we will make use of several (standard) notations and conventions that are summarized in the following list:
\begin{listing}
\item
	For an integer $k \in \N$, we put $[k] := \{1, \dots, k\}$.
\item
	Let $\x = (x_1, \dots, x_n) \in \R^n$ be a vector. The \define{support} of $\x$ is defined by the set of its non-zero components $\supp(\x) := \{ i \in [n] \suchthat x_i \neq 0 \}$ and we put $\lnorm{\x}[0] := \cardinality{\supp(\x)}$. For $1 \leq p \leq \infty$, the \define{$\l{p}$-norm} is given by
	\begin{equation}
	\lnorm{\x}[p] := \begin{cases} (\sum_{i = 1}^n \abs{x_i}^p)^{1/p}, & p < \infty, \\ \max_{i \in [n]} \abs{x_i}, & p = \infty.
	\end{cases}
	\end{equation}
	The associated \define{unit ball} is denoted by $\ball[p][n] := \{ \x \in \R^n \suchthat \lnorm{\x}[p] \leq 1 \}$ and the \define{(Euclidean) unit sphere} is $S^{n-1} := \{ \x \in \R^n \suchthat \lnorm{\x} = 1 \}$. The \define{operator norm} of a matrix $\vec{M} \in \R^{n'\times n}$ is defined as $\norm{\vec{M}} := \sup_{\x \in S^{n-1}} \lnorm{\vec{M}\x}$.
\item
	For a subset $A \subset \R^n$, we define the \define{step function} (or \define{characteristic function}) 
	\begin{equation}
	\indset{A}(\x) := \begin{cases} 1, & \x \in A, \\ 0, & \text{otherwise},
	\end{cases}\quad \x \in \R^n.
	\end{equation}
\item
	The \define{expected value} of a random variable $Z$ is denoted by $\mean[Z]$. 
	Similarly, the \define{probability} of an event $A$ is denoted by $\prob[A]$ and for the corresponding \define{indicator function}, we write $\indprob{A}$. 
\item
	The letter $C$ is always reserved for a constant, and if necessary, an explicit dependence on a certain parameter is emphasized by a subscript.
	Moreover, we refer to $C$ as a \define{numerical} constant (or \define{universal} constant) if its value is independent from all involved parameters in the current setup.
	In this case, we sometimes simply write $A \lesssim B$ instead of $A \leq C \cdot B$, and if $C_1 \cdot A \leq B \leq C_2 \cdot A$ for numerical constants $C_1, C_2 > 0$, we use the abbreviation $A \asymp B$.
\item
	The phrase ``with high probability'' means that an event arises at least with a fixed (and high) \emph{probability of success}, for instance, 99\%. Alternatively, one could regard this probability as an additional input parameter which would then also appear as a factor somewhere in the statement. But we will usually omit this explicit quantification for the sake of convenience.
\end{listing}

\section{Main Results}
\label{sec:results}

\subsection{General Signal Sets and Local Mean Width}
\label{subsec:results:signal}

In the following subsections, we aim to establish a more refined versions of our initial recovery result Theorem~\ref{thm:intro:recovery}.
For this purpose, let us first recall the error bound \eqref{eq:intro:recovery:error} of Theorem~\ref{thm:intro:recovery}. There exist some situations in which the decay rate of $O(m^{-1/4})$ is in fact non-optimal (see \cite[Sec.~4]{plan2014highdim}).
The actual reason for this drawback is the inability of the \emph{global} mean width to capture the ``local structure'' of $\sset$ around the signal-of-interest $\grtrmu$ (see also Subsection~\ref{subsec:consequences:convexgeo}).
Fortunately, it has turned out in \cite{plan2014highdim,plan2015lasso} that this issue can be resolved by the following \define{localized} version of the mean width:
\begin{definition}
The \define{local mean width} of a set $L \subset \R^n$ at scale $t \geq 0$ is defined by
\begin{equation}\label{eq:results:signal:localmeanwidth}
	\meanwidth[t]{L} := \meanwidth{L \intersec t \ball[2][n]} = \mean[\sup_{\x \in L \intersec t \ball[2][n]} \sp{\gaussian}{\x}],
\end{equation}
where $\gaussian \distributed \Normdistr{\vec{0}}{\I}$ is a standard Gaussian random vector. Moreover, we set $\effdim[t]{L} := \meanwidth[t]{L}^2 / t^2$ for $t > 0$, which is called the \define{(local) effective dimension} of $L$ at scale $t$.
\end{definition}
If $L = \sset - \grtrmu$, the supremum in \eqref{eq:results:signal:localmeanwidth} is only taken over a (small) ball of radius $t$ around $\grtrmu$. Thus, $\meanwidth[t]{\sset - \grtrmu}$ indeed measures the complexity of $\sset - \grtrmu$ in a local manner. The choice of the scaling parameter $t$ will become crucial in Theorem~\ref{thm:results:recovery} because it essentially serves as the desired \define{approximation accuracy}, in the sense that the error distance $\lnorm{\solu - \grtrmu}$ is eventually bounded by $t$. Finally, we would like to emphasize that the local mean width can be always bounded by its global counterpart:
\begin{equation}\label{eq:results:signal:localglobal}
\meanwidth[t]{L} = \mean[\sup_{\x \in L \intersec t \ball[2][n]} \sp{\gaussian}{\x}] \leq \mean[\sup_{\x \in L} \sp{\gaussian}{\x}] = \meanwidth{L}, \quad t \geq 0.
\end{equation}

\subsection{A Recovery Guarantee Based on Restricted Strong Convexity}
\label{subsec:results:recovery}

So far, we have just assumed differentiability and convexity for $\loss$ in \ref{assump:estimator:difflip} and \ref{assump:estimator:convex}, respectively.
In order to impose further conditions on $\loss$, we first need to adapt the \define{model estimation parameters} $\scalfac$, $\modeldev$, $\modeldeveta$ from \eqref{eq:intro:modelparam} to the setup of general loss functions. This is done in the following way, where $\scalfac$ is defined as the solution of \eqref{eq:results:recovery:modelparam:scalfac}:
\begin{subequations} \label{eq:results:recovery:modelparam}
\begin{align} 
	0 = {} & \mean[\loss'(\scalfac \gaussianuniv, \fobs(\gaussianuniv)) \cdot \gaussianuniv], \label{eq:results:recovery:modelparam:scalfac} \\
	\modeldev^2 := {} & \mean[\loss'(\scalfac \gaussianuniv, \fobs(\gaussianuniv))^2], \label{eq:results:recovery:modelparam:modeldev} \\
	\modeldeveta^2 := {} & \mean[\loss'(\scalfac \gaussianuniv, \fobs(\gaussianuniv))^2 \cdot \gaussianuniv^2], \label{eq:results:recovery:modelparam:modeldeveta}
\end{align}
\end{subequations}
with $g \distributed \Normdistr{0}{1}$. Since $\scalfac$ is only implicitly given here, it is not even clear if the equation \eqref{eq:results:recovery:modelparam:scalfac} has a solution. Up to now, we can merely state that, if existent, $\scalfac$ is uniquely determined.\footnote{This is a consequence of \ref{assump:estimator:convex}. More precisely, the uniqueness of $\scalfac$ follows from writing \eqref{eq:results:recovery:modelparam:scalfac} as an integral and using the fact that $\loss'$ is non-decreasing in the first variable.}
Therefore, the reader should be aware of the following: Defining $\scalfac$ by \eqref{eq:results:recovery:modelparam:scalfac} particularly means that we \emph{postulate}  the solvability of this equation.
There are in fact ``incompatible'' pairs of loss functions and non-linearities for which $\scalfac$ does not exist, and our results are of course not applicable anymore. A typical example is given in Subsection~\ref{subsec:consequences:loss}.
Interestingly, such an issue cannot arise for the square loss functional $\losssq$, since (not very surprisingly) the definition \eqref{eq:results:recovery:modelparam} exactly coincides with \eqref{eq:intro:modelparam} in this case. 
For that reason, the purpose of the adapted model parameters is still to quantify the mismatch between the non-linear observations $y_i = \fobs(\sp{\a_i}{\grtr})$ and their noiseless linear counterpart $y_i^{\text{lin}} := \sp{\a_i}{\grtr}$.

Let us now return to the fundamental question of when a loss function is suited for signal recovery by \eqref{eq:intro:estimator}.
As already sketched in Subsection~\ref{subsec:intro:estimator}, an accurate reconstruction should be possible if the empirical loss $\lossemp[\yadv]$ is not ``too flat,'' meaning that its \emph{curvature} needs to be sufficiently large around $\grtrmu$ (see again Figure~\ref{fig:intro:convexloss}).
A common way to achieve such a nice behavior is to assume \define{strong convexity} for $\lossemp[\yadv]$. In order to give a precise definition, we introduce the linear (Taylor) approximation error of $\lossemp[\yadv]$ at $\grtrmu$:
\begin{equation}\label{eq:results:recovery:taylor}
	\losstaylor{\x}{\grtrmu}{\yadv} := \lossemp[\yadv](\x) - \lossemp[\yadv](\grtrmu) - \sp{\gradient\lossemp[\yadv](\grtrmu)}{\x - \grtrmu}, \quad \x \in \R^n,
\end{equation}
where $\gradient\lossemp[\yadv]\colon \R^n \to \R^n$ denotes the gradient of $\lossemp[\yadv]$. The empirical loss is called \define{strongly convex} (with respect to $\grtrmu$) if $\losstaylor{\x}{\grtrmu}{\yadv}$ can be bounded from below by $C \lnorm{\x - \grtrmu}^2$ for all $\x \in \R^n$ and an appropriate constant $C > 0$.
But in general, such a condition cannot be satisfied as long as $m < n$. To see this, just consider the square loss $\losssq$ and observe that (after a simple calculation) $\losssqtaylor{\x}{\grtrmu}{\yadv} = \tfrac{1}{2m} \lnorm{\A(\x - \grtrmu)}^2$. If $m < n$, the matrix $\A \in \R^{m \times n}$ has a non-trivial kernel and we have $\losssqtaylor{\x}{\grtrmu}{\yadv} = 0$ for some $\x \neq \grtrmu$.

But fortunately, the estimator \eqref{eq:intro:estimator} does only take account of a fixed subset $\sset$ of $\R^n$. 
Hence, it is actually enough to have \define{restricted strong convexity}, i.e., $\losstaylor{\x}{\grtrmu}{\yadv} \geq C \lnorm{\x - \grtrmu}^2$ only needs to hold for all $\x$ in (a subset of) $\sset$. This simple idea leads us to the following important relaxation of strong convexity, which is adapted from \cite{negahban2009unified,negahban2012unified}:
\begin{definition}\label{def:results:rsc}
The empirical loss function $\lossemp[\yadv]$ satisfies \define{restricted strong convexity} (\define{RSC}) at scale $t \geq 0$ (with respect to $\grtrmu$ and $\sset$) if there exists a constant $C > 0$ such that
\begin{equation}\label{eq:results:rsc}
	\losstaylor{\x}{\grtrmu}{\yadv} \geq C \lnorm{\x - \grtrmu}^2 \quad \text{for all $\x \in K \intersec (t \S^{n-1} + \grtrmu)$.}
\end{equation}
Then, we usually call $C$ the \define{RSC-constant} of $\loss$.
\end{definition}
Geometrically, \eqref{eq:results:rsc} states that strong convexity holds for those vectors that belong to a (small) sphere of radius $t$ around $\grtrmu$.
An interesting special case of Definition~\ref{def:results:rsc} is again obtained for $\loss = \losssq$: Here, we have $\losssqtaylor{\x}{\grtrmu}{\yadv} = \tfrac{1}{2m} \lnorm{\A(\x - \grtrmu)}^2$ and \eqref{eq:results:rsc} simply corresponds to the fact that the \define{restricted minimum singular value}\footnote{The word ``restricted'' is again referred to the assumption that $\x - \grtrmu \in (K - \grtrmu) \intersec t \S^{n-1}$.} of $\tfrac{1}{m}\A$ is bounded from below by a positive constant. 
This property is precisely reflected by \define{Gordon's escape lemma}, which has formed a key ingredient in the proofs of recovery results with Lasso-type estimators; see \cite[Thm.~4.2, Lem.~4.4]{plan2015lasso} and \cite[Thm.~3.2]{chandrasekaran2012geometry} for example. 
The following refinement of Theorem~\ref{thm:intro:recovery} shows that the concept of RSC indeed allows us to incorporate general loss functions: 
\begin{theorem}\label{thm:results:recovery}
Suppose that the assumptions \ref{assump:model:nonlinmod}, \ref{assump:model:advnoise}, \ref{assump:estimator:difflip}, \ref{assump:estimator:convex} hold with $\Covmatr = \I$, and that $\scalfac$, $\modeldev$, $\modeldeveta$ are defined according to \eqref{eq:results:recovery:modelparam}. Moreover, let $\lnorm{\grtr} = 1$ and $\grtrmu \in \sset$ for a convex set $\sset \subset \R^n$.
For a fixed number $t > 0$, we assume that the empirical loss function $\lossemp[\yadv]$ satisfies RSC at scale $t$ (with respect to $\grtrmu$ and $\sset$) and that 
\begin{equation}\label{eq:results:recovery:measurements}
m \geq \effdim[t]{\sset - \grtrmu}.
\end{equation}
Then, there exists a constant $C > 0$ such that the following holds with high probability:\footnote{More precisely, the constant $C$ only depends on (the RSC-constant of) $\loss$ and the probability of success.} If
\begin{equation}\label{eq:results:recovery:accuracy}
	t > C\cdot \accuracy := C\cdot \Big( \frac{\modeldev \cdot \sqrt{\effdim[t]{\sset - \grtrmu}} + \modeldeveta}{\sqrt{m}} + \advdev \Big),
\end{equation}
any minimizer $\solu$ of \eqref{eq:intro:estimator} satisfies $\lnorm{\solu - \grtrmu} \leq t$.
\end{theorem}

This result looks slightly more technical than Theorem~\ref{thm:intro:recovery}. 
Hence, it is very helpful to take the following perspective when interpreting the statement of Theorem~\ref{thm:results:recovery}: Suppose that RSC and \eqref{eq:results:recovery:measurements} are satisfied for a fixed (small) accuracy $t > 0$. 
Then, in order to achieve $\lnorm{\solu - \grtrmu} \leq t$ (with high probability), one simply needs to ensure that $t > C \cdot \accuracy$. 
But the size of $\accuracy$ can be easily controlled by the sample count $m$ as well as by the adversarial noise parameter $\advdev$ (cf. \eqref{eq:intro:advnoise}). Thus, one might adjust $m$ such that $t \approx C \cdot \accuracy$ holds---this actually corresponds to the \emph{minimal} number of required observations to invoke Theorem~\ref{thm:results:recovery}.
In that case, we obtain
\begin{equation}
\lnorm{\solu - \grtrmu} \leq t \approx C \cdot \accuracy = C\cdot \Big( \frac{\modeldev \cdot \sqrt{\effdim[t]{\sset - \grtrmu}} + \modeldeveta}{\sqrt{m}} + \advdev \Big),
\end{equation}
which resembles the formulation of Theorem~\ref{thm:intro:recovery}. Nevertheless, one should be aware of the fact that the desired accuracy is only implicitly coupled with $m$ here, since the scale $t$ has been fixed in advance.
Apart from that analytical description, there is also a nice geometric interpretation of Theorem~\ref{thm:results:recovery} which is presented in Subsection~\ref{subsec:consequences:convexgeo}. The related discussion will particularly illustrate how the local effective dimension $\effdim[t]{\cdot}$ varies as a function of $t$.

Let us now analyze the parameter $\accuracy$, which certainly plays the role of a lower bound for the achievable error accuracy. As already seen before, the adversarial noise is incorporated by an ``inevitable'' term $\advdev$. 
On the other hand, the structure of the signal set is captured by the local mean width $\meanwidth[t]{\sset - \grtrmu}$, whereas the model parameters $\scalfac$, $\modeldev$, $\modeldeveta$ reflect our uncertainty about the true output rule.
This observation is remarkable because it indicates that all components of our framework can be essentially handled by separate quantities, and even more importantly, the asymptotic error rate $O(m^{-1/2})$ is not affected by them. Consequently, the capability of recovering signals does, at least asymptotically, neither depend on the (noisy) non-linearity $\fobs$ nor on the specific choice of loss.

However, when studying the precise quantitative behavior of the above error estimate, the impact of $\scalfac$, $\modeldev$, and $\modeldeveta$ becomes much more significant. 
On the one hand, we aim for small values of $\modeldev$ and $\modeldeveta$, which quantify the ``variance'' of the model mismatch. 
And on the other hand, a bad scaling of $\grtrmu$, that is, $\abs{\scalfac} \approx 0$, should be avoided because any bound for $\lnorm{\solu - \grtrmu}$ would be almost meaningless otherwise.
If $\loss = \losssq$, there is also an easy statistical explanation for this drawback: In this situation, $\scalfac = \mean[\fobs(\gaussianuniv) \cdot \gaussianuniv]$ essentially measures the \emph{correlation} between the non-linear and linear model. Therefore, $\abs{\scalfac} \approx 0$ would imply that $y_i = \fobs(\sp{\a_i}{\grtr})$ and $y_i^{\text{lin}} = \sp{\a_i}{\grtr}$ are almost uncorrelated, so that there is no hope for recovering $\grtr$ anyway.

In general, we can now make precise our initial heuristic that the used loss function and the observation model should ``play well together'':
A preferable setup for Theorem~\ref{thm:results:recovery} is given when both $\modeldev$ and $\modeldeveta$ are relatively small compared to $\abs{\scalfac}$, i.e., $\abs{\scalfac} / \max\{\modeldev, \modeldeveta\} \gg 1$. 
Interestingly, we shall see in the course of Example~\ref{ex:consequences:model}\ref{ex:consequences:model:linear} that the latter quotient can be also regarded as the \define{signal-to-noise ratio} of the output scheme \ref{assump:model:nonlinmod}. 

\begin{remark}
\begin{rmklist}
\item
	Although not stated explicitly, the proof of Theorem~\ref{thm:results:recovery} already implies that there exists a minimizer of \eqref{eq:intro:estimator}. More specifically, this follows from the assumption that the equations in \eqref{eq:results:recovery:modelparam} are well-defined and $\scalfac$ actually exists.
\item
	The unit-norm assumption on the ground-truth signal $\grtr$ cannot be dropped in general. For example, one might consider a binary observation model $y_i = \sign(\sp{\a_i}{\grtr})$. Then, we could multiply $\grtr$ by any positive scalar without changing the value of $y_i$. 
	Thus, there is no chance to recover the magnitude of $\grtr$ from $\y$ and $\A$, which in turn makes a certain normalization necessary. 
	In the case of a (noisy) linear model, however, one could easily avoid such an additional constraint by appropriately scaling the signals.
	
	Furthermore, the non-linearity $\fobs$ in \ref{assump:model:nonlinmod} might involve some rescaling of the linear projection $\sp{\a_i}{\grtr}$. This effect is precisely captured by the parameter $\scalfac$, which particularly explains the condition $\grtrmu \in \sset$ in Theorem~\ref{thm:intro:recovery} and Theorem~\ref{thm:results:recovery} (instead of $\grtr \in \sset$). Equivalently, one can also assume that $\grtr \in \scalfac^{-1} \sset$, corresponding to a dilation of the signal set $\sset$.
	But even though this is just a constant factor, such an issue might become very relevant in practice, as $\scalfac$ could be unknown.
	Due to $\lnorm{\grtr} = 1$, it is sometimes necessary to enlarge\footnote{By ``enlarging,'' we simply mean that $\sset$ is dilated by some scalar $\lambda > 1$, i.e., $\sset \mapsto \lambda \sset$.} $\sset$ in order to guarantee $\scalfac\grtr \in \sset$. But the effective dimension might grow at the same time, so that the error bound of Theorem~\ref{thm:results:recovery} becomes worse. 
	This phenomenon is basically equivalent to finding an optimal tuning parameter for a regularized optimization program, such as the classical Lasso \eqref{eq:intro:lasso} with $R \geq 0$.
\item
	If there exists no adversarial noise, i.e., $\y = \yadv$, the additive error term $\advdev$ in $\accuracy$ vanishes. Moreover, a careful analysis of the proof of Theorem~\ref{thm:results:recovery} shows that the assumption of Lipschitz continuity in \ref{assump:estimator:difflip} as well as \eqref{eq:results:recovery:measurements} can be dropped in this case.
	
	When additionally assuming that $\loss = \losssq$, we can easily reproduce the original recovery result for the $\sset$-Lasso \eqref{eq:intro:klasso} presented in \cite[Thm.~1.9]{plan2015lasso}. Note that the RSC holds for \emph{any} scale $t > 0$ here (with high probability), which is a consequence of a version of Gordon's escape lemma (cf. \cite[Lem.~4.4]{plan2015lasso}).
	In this situation, we can draw the remarkable conclusion that the estimator \eqref{eq:intro:klasso} with non-linear inputs has essentially the same behavior as if it would have been applied to a simple linear model of the form $y_i = \scalfac \sp{\a_i}{\grtr} + z_i$ with $z_i \distributed \Normdistr{0}{\modeldev^2}$.
	\qeddiamond
\end{rmklist}
\end{remark}

\subsection{A Sufficient Condition for Restricted Strong Convexity}
\label{subsec:results:rsc}

In the previous part, it has turned out that RSC is a key concept to achieve recovery guarantees. 
Now, we shall investigate some sufficient conditions for this property.
The following theorem shows that RSC holds under fairly general assumptions on the considered loss:
\begin{theorem}\label{thm:results:rsc}
Suppose that \ref{assump:model:nonlinmod}, \ref{assump:model:advnoise}, \ref{assump:estimator:difflip}, \ref{assump:estimator:convex} hold with $\Covmatr = \I$. Furthermore, let $\lnorm{\grtr} = 1$ and $\grtrmu \in \sset$ for a (not necessarily convex) set $\sset \subset \R^n$.
For any fixed scale $t > 0$, we assume that the loss function $\loss \colon \R \times \Y \to \R$ satisfies the following properties:
\begin{properties}[3em]{L}
\setcounter{enumi}{2}
\item\label{thm:results:rsc:regularity}
	$\loss$ is twice continuously differentiable in the first variable. The corresponding second derivative is denoted by $\loss''(v,y) = \partder{\loss}{v}[2](v,y)$.
\item\label{thm:results:rsc:strongconvexity}
	Let us fix two numbers $0 < M, N \leq \infty$ with $M \geq \max\{32t, C_1 \cdot \abs{\scalfac}\}$ and $N \geq C_2 \cdot \mean[\abs{\fobs(\gaussianuniv)}] + \lnorm{\y - \yadv}[\infty]$, where $\gaussianuniv \distributed\Normdistr{0}{1}$ and $C_1, C_2 > 0$ are numerical constants. We assume that there exists a constant $C_{M,N} > 0$ such that
	\begin{equation}
		\loss''(v,y) \geq C_{M,N} \quad \text{for all $(v,y) \in \intvop{-M}{M} \times (\intvop{-N}{N} \intersec \Y)$.}
	\end{equation}
\end{properties}
\vspace{-.75\baselineskip}
Then, we can find numerical constants $C_3, C_4 > 0$ such that the following holds with probability at least $1-\exp(-C_4\cdot m)$:
If the number of observations obeys
\begin{equation}\label{eq:results:rsc:meanwidth}
	m \geq C_3 \cdot \effdim[t]{K-\grtrmu},
\end{equation}
then the empirical loss function $\lossemp[\yadv]$ satisfies RSC at scale $t$ (with respect to $\grtrmu$ and $\sset$), i.e., there exists $C_\RSC > 0$ (only depending on $C_{M,N}$) such that
	\begin{equation}\label{eq:results:rsc:rsc}
		\losstaylor{\x}{\grtrmu}{\yadv} \geq C_\RSC \lnorm{\x - \grtrmu}^2 \quad \text{for all $\x \in K \intersec (t \S^{n-1} + \grtrmu)$.}
	\end{equation}
\end{theorem}

A proof of Theorem~\ref{thm:results:rsc} is given in Subsection~\ref{subsec:proof:rsc}.
Although \ref{thm:results:rsc:strongconvexity} looks rather technical, this is actually a relatively mild condition on $\loss$. Roughly speaking, it requires that for ``sufficiently'' large values of $M$ and $N$ (note that $t$ is typically assumed to be very small), one can bound $\loss''$ from below by a step function supported around the origin:
\begin{equation}
	\loss''(v,y) \geq C_{M,N} \cdot \indset{\intvop{-M}{M}\times\intvop{-N}{N}}(v, y) \quad \text{for all $(v, y)\in \R\times \Y$.}
\end{equation}
In other words, we assume that $\loss$ is \emph{locally} strongly convex in an a certain neighborhood of $(0,0)$.
For example, this is automatically satisfied for \emph{strictly convex} loss functions, which were already considered in Theorem~\ref{thm:intro:recovery}:
\begin{corollary}\label{cor:results:rsc:strictlyconvex}
The assumption \ref{thm:results:rsc:strongconvexity} in Theorem~\ref{thm:results:rsc} is fulfilled if $\loss''(v,y) \geq \mathcal{F}(v) > 0$ holds for all $(v,y) \in \R \times \Y$ and some continuous function $\mathcal{F}\colon \R \to \intvop{0}{\infty}$. 
\end{corollary}
\begin{proof}
	Choose $M := \max\{32t, C_1 \cdot \abs{\scalfac}\}$ and $N := \infty$. The continuity of $\mathcal{F}$ implies that a positive minimum $C_{M,N}$ of $\mathcal{F}$ is attained in the interval $\intvcl{-M}{M}$, and therefore $\loss''(v,y) \geq \mathcal{F}(v) \geq C_{M,N}$ for all $v \in \intvop{-M}{M}$ and $y \in \Y$.
\end{proof}

\begin{remark}\label{rmk:results:rsc:rscconstant}
\begin{rmklist}
\item\label{rmk:results:rsc:rscconstant:scale}
	It is very important to notice that the RSC-constant $C_\RSC$ in \eqref{eq:results:rsc:rsc} only depends on the choice of $M$ and $N$. In fact, when combining this statement with Theorem~\ref{thm:results:recovery} (as we do in the next subsection), we need to ensure that the constant $C$ in \eqref{eq:results:recovery:accuracy} has no explicit dependence on $t$. 
	Certainly, the condition of \ref{thm:results:rsc:strongconvexity} requires that $M \geq 32 t$, but this is no severe restriction, since we are interested in small values of $t$ anyway. 
	A careful analysis of the proof of Theorem~\ref{thm:results:rsc} shows that $C_\RSC$ scales as $C_{M,N}$. 
	Hence, if $M$ and $N$ become large, the size of the RSC-constant actually relies on the asymptotic decay of $(v,y) \mapsto \loss''(v,y)$.
	However, Theorem~\ref{thm:results:rsc} is rather a non-asymptotic result, since $M$ and $N$ are fixed albeit possibly large.
	In some situations, it might be even necessary to set $N = \infty$, for example, if the term $\lnorm{\y - \yadv}[\infty] \ (\geq \advdev)$ is a random variable for which no upper bound (with high probability) exists.
\item
	When using the square loss $\losssq(v,y) = \tfrac{1}{2} (v-y)^2$, the conditions \ref{assump:estimator:difflip}--\ref{thm:results:rsc:strongconvexity} are trivially fulfilled with $C_{M,N} = 1$, since $(\losssq)'' \equiv 1$. We have already pointed out that, in this case, RSC is equivalent to having a lower bound for the minimal restricted singular value of $\tfrac{1}{m}\A$. 
	In that sense, one can regard Theorem~\ref{thm:results:rsc} as a natural generalization of classical concepts, such as \emph{Gordon's escape through a mesh} (\cite{gordon1988escape}) or \emph{restricted eigenvalues} (see \cite{raskutti2010restrev,geer2009lasso} for example).
	\qeddiamond
\end{rmklist}
\end{remark}

\subsection{Extension to Anisotropic Measurement Vectors}
\label{subsec:results:aniso}

In this final part, we shall drop the hypothesis of an isotropic measurement process, that means, we now allow the vectors $\a_1, \dots, \a_m$ to follow any multivariate Gaussian distribution $\Normdistr{\vec{0}}{\Covmatr}$ with arbitrary positive definite covariance matrix $\Covmatr \in \R^{n \times n}$. The following theorem shows that, with some slight modifications of the previous results, signal recovery is still feasible:\footnote{If $\Covmatr \in \R^{n \times n}$ is positive definite, there exists a unique, positive definite \define{matrix square root} $\sqrt{\Covmatr} \in \R^{n \times n}$ with $\sqrt{\Covmatr} \cdot \sqrt{\Covmatr} = \Covmatr$.}
\begin{theorem}\label{thm:results:aniso}
Suppose that \ref{assump:model:nonlinmod}, \ref{assump:model:advnoise}, \ref{assump:estimator:difflip}--\ref{thm:results:rsc:strongconvexity} hold true, and that $\scalfac$, $\modeldev$, $\modeldeveta$ are defined according to \eqref{eq:results:recovery:modelparam}. Moreover, let $\lnorm{\sqrt{\Covmatr}\grtr} = 1$ and $\grtrmu \in \sset$ for a convex set $\sset \subset \R^n$.
Then, there exist constants\footnote{More specifically, $C_1$ is numerical constant, whereas $C_2$ might depend on the ``probability of success'' as well as on the choice of $C_{M,N}$ in \ref{thm:results:rsc:strongconvexity}; see also Remark~\ref{rmk:results:rsc:rscconstant}\ref{rmk:results:rsc:rscconstant:scale}.} $C_1, C_2 > 0$ such that the following holds for any error accuracy $t > 0$: Assuming that the number of observations obeys
\begin{equation}
	m \geq C_1 \cdot \effdim[t]{\sqrt{\Covmatr}(\sset - \grtrmu)},
\end{equation}
and that
\begin{equation}
	t > C_2\cdot \accuracy := C_2\cdot \Big( \frac{\modeldev \cdot \sqrt{\effdim[t]{\sqrt{\Covmatr}(\sset - \grtrmu)}} + \modeldeveta}{\sqrt{m}} + \advdev \Big),
\end{equation}
any minimizer $\solu$ of \eqref{eq:intro:estimator} satisfies $\lnorm{\sqrt{\Covmatr}(\solu - \grtrmu)} \leq t$ with high probability.
\end{theorem}

\begin{proof}
The statement of Theorem~\ref{thm:results:aniso} for $\Covmatr = \I$ is a straightforward combination of our main results, Theorem~\ref{thm:results:recovery} and Theorem~\ref{thm:results:rsc}.
Next, we follow the idea of \cite[Cor.~1.6]{plan2015lasso}, where the case of an anisotropic covariance structure has been reduced to the situation of $\Covmatr = \I$.
For this, we first observe that one may write $\a_i = \sqrt{\Covmatr}\a_i'$ for some $\a_i' \distributed \Normdistr{\vec{0}}{\I}$. This implies
\begin{equation}
\sp{\a_i}{\x} = \sp{\sqrt{\Covmatr} \a_i'}{\x} = \sp{\a_i'}{\sqrt{\Covmatr}\x} \quad \text{for all $\x \in \R^n$.}
\end{equation}
Hence, substituting $\grtr$ by $\sqrt{\Covmatr}\grtr$ and $\A$ by $\A' := \matr{\a_1' & \dots & \a_m'}^\T$, we obtain precisely the same model as in \ref{assump:model:nonlinmod} but with a trivial covariance matrix $\Covmatr = \I$.
The generalized estimator \eqref{eq:intro:estimator} can be treated in a similar way:
\begin{align}
	\solu ={}& \argmin_{\x \in \sset} \tfrac{1}{m} \sum_{i = 1}^m \loss(\sp{\a_i}{\x}, \tilde{y}_i) \\
	={}& \argmin_{\x \in \sset} \tfrac{1}{m} \sum_{i = 1}^m \loss(\sp{\a'_i}{\sqrt{\Covmatr}\x}, \tilde{y}_i) \\
	={}& \sqrt{\Covmatr}^{-1} \cdot \argmin_{\x' \in \sqrt{\Covmatr}\sset} \tfrac{1}{m} \sum_{i = 1}^m \loss(\sp{\a_i'}{\x'}, \tilde{y}_i).
\end{align}
Now, we are exactly in the initial setup of Theorem~\ref{thm:results:aniso} with $\Covmatr = \I$, but with the caveat that $\solu$, $\grtr$, and $\sset$ are replaced by $\sqrt{\Covmatr}\solu$, $\sqrt{\Covmatr}\grtr$, and $\sqrt{\Covmatr}\sset$, respectively. Since the statement has been already proven for $\Covmatr = \I$, the general claim follows.
\end{proof}

A crucial feature of Theorem~\ref{thm:results:aniso} is that it still works with the original estimator \eqref{eq:intro:estimator}, which does not require any knowledge of $\Covmatr$. 
This implies a certain practical relevance, since the exact correlation structure of $\A$ is mostly unknown in real-world applications.
But as before, we need to ``pay the price'' that the actual recovery statement explicitly depends on our model uncertainty:
On the one hand, the condition of $\lnorm{\sqrt{\Covmatr} \grtr} = 1$ might involve a certain rescaling of the ground-truth signal $\grtr$, and the error bound $\lnorm{\sqrt{\Covmatr}(\solu - \grtrmu)} \leq t$ is affected by $\sqrt{\Covmatr}$ as well. And on the other hand, we need to control the modified mean width $\meanwidth[t]{\sqrt{\Covmatr}(\sset - \grtrmu)}$. This could be done similarly to \cite[Rmk.~1.7]{plan2015lasso}, leading to a (non-optimal) estimate\footnote{Note that in contrast to \cite[Rmk.~1.7]{plan2015lasso}, we do not assume that $\sset - \grtrmu$ is a cone, which makes the estimate slightly weaker.}
\begin{equation}
\meanwidth[t]{\sqrt{\Covmatr}(\sset - \grtrmu)} \leq \meanwidth[t]{\sset - \grtrmu} \cdot \begin{cases}
	\norm{\sqrt{\Covmatr}}, & \text{if $\norm{\sqrt{\Covmatr}^{-1}} < 1$,}\\
	\norm{\sqrt{\Covmatr}^{-1}}\norm{\sqrt{\Covmatr}}, & \text{if $\norm{\sqrt{\Covmatr}^{-1}} \geq 1$.}\\
\end{cases}
\end{equation}
Informally stated, we obtain the following final conclusion from Theorem~\ref{thm:results:aniso}:
\begin{highlight}
	If the covariance matrix $\Covmatr$ is not too badly conditioned, an accurate recovery of $\grtr$ by \eqref{eq:intro:estimator} is still possible if $m \gg \effdim[t]{\sset - \grtrmu}$ and the adversarial noise is low.
\end{highlight}

\begin{remark}
	By applying exactly the same strategy as in the proof of Theorem~\ref{thm:results:aniso}, one could also generalize the statements of Theorem~\ref{thm:results:recovery} and Theorem~\ref{thm:results:rsc} to anisotropic measurements. But note that the property of RSC then needs to be formulated with respect to $\sqrt{\Covmatr} \grtrmu$ and $\sqrt{\Covmatr} \sset$. These technical details are left to the reader. \qeddiamond
\end{remark}

\section{Consequences and Examples}
\label{sec:consequences}

The theoretical outcomes of the previous section offer a broad range of applications to compressed sensing, signal processing as well as machine learning.
To demonstrate this versatility, we shall now discuss several examples of signal sets, observation models, and loss functions, which could be applied to our main results.

\subsection{Examples of Signal Sets}
\label{subsec:consequences:signal}

Let us recall the crucial assumption
\begin{equation}\label{eq:consequences:signal:meascount}
m \gtrsim \effdim[t]{\sset - \grtrmu} = \meanwidth[t]{\sset - \grtrmu}^2 / t^2,
\end{equation}
which appears both in Theorem~\ref{thm:results:recovery} and Theorem~\ref{thm:results:rsc}.
It implies that the number of required measurements heavily relies on the effective dimension of $\sset - \grtrmu$.
This makes especially those sets of small mean width appealing, coinciding with our wish to exploit the low-complexity structure of the considered signal class.
On the other hand, it is usually very hard to compute $\effdim[t]{\sset - \grtrmu}$ for an arbitrary subset $\sset \subset \R^n$, or at least, to find a meaningful upper bound for it.
The notion of mean width (or effective dimension) is therefore rather of theoretical interest in a general situation. But fortunately, there exist many important special cases for which (sharp) bounds are available. For the following examples, let us recall that the \emph{global} effective dimension is given by $\effdim{L} := \meanwidth{L}^2$ for $L \subset \R^n$.
\begin{example}\label{ex:consequences:signal}
\begin{rmklist}
\item
	\define{Linear subspaces.} Assume that $\sset \subset \R^n$ is a linear subspace of dimension $d$. If $\grtrmu \in \sset$, we have (cf. \cite{plan2015lasso})
	\begin{equation}\label{eq:consequences:signal:linsubspace}
		\effdim[t]{\sset - \grtrmu} = \effdim[1]{\sset - \grtrmu} = \effdim{(\sset - \grtrmu) \intersec \ball[2][n]} \asymp d.
	\end{equation}
	In fact, $\effdim{\cdot}$ and $\effdim[t]{\cdot}$ measure the algebraic dimension here, which particularly justifies why we speak of the \emph{effective dimension} of a set. Another consequence of \eqref{eq:consequences:signal:linsubspace} is a trivial upper bound for the effective dimension of any subset $\sset \subset \R^n$:
	\begin{equation}
		\effdim[t]{\sset - \grtrmu} \leq \effdim[t]{\R^n} \asymp n.
	\end{equation}
	With regard to our main results, this mimics the behavior of classical, unconstrained estimators, such as an ordinary least-squares fit.
\item\label{ex:consequences:signal:polytope}
	\define{Finite sets and polytopes.} Let $\sset' = \{\x_1,\dots,\x_k\} \subset \R^n$ be a finite set. Then, $\sset := \convhull(\sset')$ is a polytope and we have (cf. \cite[Ex.~1.3.8]{vershynin2014estimation})
	\begin{equation}
		\effdim{\sset - \grtrmu} = \effdim{\sset' - \grtrmu} \lesssim (\max_{1 \leq i \leq k}\lnorm{\x_i - \grtrmu}^2) \cdot \log(k),
	\end{equation}
	where we have also used that the mean width is invariant under taking the convex hull (\cite[Prop.~2.1]{plan2013robust}). This shows that the effective dimension of a polytope only logarithmically depends on the number of its vertices, even though it might have full algebraic dimension.
\item
	\define{(Approximately) Sparse vectors.} The reconstruction of sparse signals is one of the ``driving forces'' for our results. Let us assume that $\grtr \in S^{n-1}$ is \define{$s$-sparse}, i.e., $\lnorm{\grtr}[0] \leq s$. From the Cauchy-Schwarz inequality, it follows $\lnorm{\grtr}[1] \leq \sqrt{\lnorm{\grtr}[0]} \cdot \lnorm{\grtr} \leq \sqrt{s}$, that is, $\grtr \in \sqrt{s}\ball[1][n]$. Therefore, $\sqrt{s}\ball[1][n] \intersec S^{n-1}$ can be seen as a set of \define{approximately sparse vectors}. It was shown in \cite[Sec.~III]{plan2013robust} that for $\sset = \scalfac (\sqrt{s}\ball[1][n] \intersec \ball[2][n])$, one has 
	\begin{equation}
		\effdim{\sset - \grtrmu} \lesssim s \log(\tfrac{2n}{s}).
	\end{equation}
	The same bound can be achieved for $\effdim[t]{\sset - \grtrmu}$ if $\sset$ is an appropriately scaled $\l{1}$-ball such that $\grtrmu$ lies at its boundary (cf. \cite{chandrasekaran2012geometry,plan2015lasso}). Putting this estimate into \eqref{eq:consequences:signal:meascount}, we essentially end up with a threshold for the minimal number of required measurements that is well-known from the theory of compressed sensing and sparse recovery.
\item
	\define{(Sparse) Representations in dictionaries.} In many real-world scenarios, the hypothesis of (approximate) sparsity only holds with respect to a certain \define{dictionary} $\dict \in \R^{n \times n'}$, i.e., there exists a sparse coefficient vector $\x' \in \R^{n'}$ such that $\grtr = \dict \x'$. More generally, we may assume that $\sset = \dict \sset'$ for some appropriate \define{coefficient set} $\sset' \subset \R^{n'}$. A basic application of \define{Slepian's inequality} \cite[Lem.~8.25]{foucart2013cs} yields
	\begin{equation}\label{eq:consequences:signal:slepianbound}
		\effdim{\sset - \grtrmu} = \effdim{\dict(\sset' - \scalfac \x')} \lesssim \norm{\dict}^2 \cdot \effdim{\sset' - \scalfac \x'},
	\end{equation}
	implying that the impact of a linear transformation can be controlled by means of its operator norm. But note that the estimate of \eqref{eq:consequences:signal:slepianbound} might become very weak when $\norm{\dict}$ is large, which is typically the case for overcomplete dictionaries. However, there are often sharper bounds available.
	For example, if $\sset' = \convhull(\{\x_1',\dots,\x_k'\}) \subset \R^{n'}$ is a polytope, so is $\sset = \dict \sset'$, and by Example~\ref{ex:consequences:signal}\ref{ex:consequences:signal:polytope}, we obtain
	\begin{equation}
		\effdim{\sset - \grtrmu} = \effdim{\dict(\sset' - \scalfac \x')} \lesssim (\max_{1 \leq i \leq k}\lnorm{\dict(\x_i' - \scalfac \x')}^2) \cdot \log(k).
	\end{equation}	
\item
	There are many further examples of structured and low-dimensional signal sets. For instance, one could also consider the \define{elastic net} \cite{zou2005elastic}, \define{fused penalties} \cite{tibshirani2005fused}, \define{OSCAR/OWL} \cite{bondell2008oscar,figueiredo2014owl}, etc.
	\qeddiamond
\end{rmklist}
\end{example}
In some of the above examples, we have only given estimates of the global mean width. In fact, its local counterpart $\meanwidth[t]{\sset - \grtrmu}$ is usually harder to control because it particularly depends on the ``position'' of $\grtrmu$ in $\sset$.
We will return to this issue in Subsection~\ref{subsec:consequences:convexgeo}, where the convex geometry of our problem setup is discussed.

\subsection{Examples of (Non-)Linear Models}
\label{subsec:consequences:model}

Now, let us give some typical examples of the (random) non-linearity $\fobs\colon \R \to \Y$ in our observation model \ref{assump:model:nonlinmod}:
\begin{example}\label{ex:consequences:model}
\begin{rmklist}
\item\label{ex:consequences:model:linear}
	\define{Noisy linear observations.} In the classical settings of compressed sensing and linear regression, one considers $f(v) = \lambda v + z$, where $\lambda > 0$ is fixed and $z$ is mean-zero noise (independent of $\A$). Since the randomness of $\fobs$ is understood sample-wise, this leads to observations of the form $y_i = \lambda \sp{\a_i}{\grtr} + z_i$, or in short $\y = \lambda\A\grtr + \z$ with $\z = (z_1, \dots, z_m)$. If $\loss = \losssq$, we can explicitly compute the model parameters (cf. \eqref{eq:intro:modelparam} and \eqref{eq:results:recovery:modelparam}):
	\begin{equation}
		\scalfac = \mean[\fobs(\gaussianuniv) \cdot \gaussianuniv] = \lambda, \quad \modeldev^2 = \mean[(\fobs(\gaussianuniv) - \scalfac\gaussianuniv)^2] = \mean[z^2], \quad \modeldeveta^2 = \mean[(\fobs(\gaussianuniv) - \scalfac\gaussianuniv)^2 \cdot \gaussianuniv^2] = \mean[z^2].
	\end{equation}
	As our intuition suggests, $\scalfac$ indeed measures the contribution (scaling) of the signal part of $\fobs$, whereas $\modeldev^2$ and $\modeldeveta^2$ capture the variance of the noise. 
	Hence, one can regard the quotient $\scalfac / \max\{\modeldev, \modeldeveta\}$ as a certain \emph{signal-to-noise ratio} of the underlying output rule.
	Regarding Theorem~\ref{thm:intro:recovery} and Theorem~\ref{thm:results:recovery}, we can particularly conclude that the constants of the respective error bounds are well-behaved if this signal-to-noise ratio is large.
	
	An interesting special case is the \emph{noiseless} scenario, where $z \equiv 0$ and $\advdev = 0$. Then, $\modeldev^2 = \modeldeveta^2 = 0$ and our main results even provide \emph{exact recovery}, supposed that $m \gtrsim \effdim[t]{\sset - \grtrmu}$ \emph{for all} $t > 0$. The latter condition is closely related to the mean width of \define{descent cones} which is discussed in Subsection~\ref{subsec:consequences:convexgeo}.
\item\label{ex:consequences:model:onebit}
	\define{$1$-bit observations.} A typical example of a non-linearity is the $\sign$-function $\fobs(v) = \sign(v)$ with $\Y = \{-1,0,+1\}$, encoding a signal by $1$-bit measurements. Such a model is usually associated with the problems of $1$-bit compressed sensing and classification. Similarly to Example~\ref{ex:consequences:model}\ref{ex:consequences:model:linear}, we can also allow for noise here, e.g., by incorporating \define{random bit-flips}: $\fobs(v) = \eps \cdot \sign(v)$ where $\eps$ is an independent $\pm 1$-valued random variable with $\prob[\eps = 1] =: p$. If $\loss = \losssq$, we have
	\begin{align}
		\scalfac &= \mean[\eps \cdot \sign(\gaussianuniv) \cdot \gaussianuniv] = \mean[\eps] \cdot \mean[\abs{\gaussianuniv}] =  (2p - 1) \cdot \sqrt{\tfrac{2}{\pi}}, \\
		\modeldev^2 &= 1 - \tfrac{2}{\pi} (1-2p)^2, \quad \modeldeveta^2 = 1 - \tfrac{2}{\pi} (1-2p)^2.
	\end{align}
	Since $\scalfac > 0$ for any $p > 1/2$, this shows that $\grtrmu$ can be still estimated even when the chance of a bit-flip is close to $\frac{1}{2}$ (\cite[Sec.~III.A]{plan2013robust}).
	Another important class of observation models takes the form $\fobs(v) = \sign(v + z)$, where $z$ is an independent noise term, e.g., following a logit distribution. This corresponds to the classical setting of \define{logistic regression}, which has been widely studied in statistical literature, also in combination with sparsity-promoting constraints (see \cite{bach2010logreg} for instance).
	
	Apart from independent noise, one could also permit adversarial bit-flips according to \ref{assump:model:advnoise}. In this case, $\yadv \in \{-1,+1\}^m$ is a set of binary responses and $\advfrac := \lnorm{\y - \yadv}[0]$ counts the number of bit errors. This implies that the adversarial noise parameter
	\begin{equation}
		\advdev = \Big(\tfrac{1}{m} \sum_{i = 1}^m \abs{\tilde{y}_i - y_i}^2\Big)^{1/2} = 2 \cdot \sqrt{\tfrac{\advfrac}{m}}
	\end{equation}
	is proportional to the (square root of the) \emph{fraction of wrong observations}.
\item
	Besides these two examples, there are various other observation models that one could think of, e.g., \define{generalized linear models},\footnote{However, one should be aware of the fact that there is no strict containment (in both directions) between the single-index model of \ref{assump:model:nonlinmod} and generalized linear models (cf. \cite[Sec.~6]{plan2014highdim}).} \define{quantized measurements}, or \define{multiple labels}.
	\qeddiamond
\end{rmklist}
\end{example}

\subsection{Examples of Convex Loss Functions}
\label{subsec:consequences:loss}

The recovery bound of Theorem~\ref{thm:results:recovery} indicates that the asymptotic error rate (in $m$) is essentially independent from the applied loss function.
However, the precise quantitative behavior of the approximation error will substantially depend on the choice of $\loss$ as well as the underlying observation model.
In this work, we do not discuss the issue of what loss function is best suited for a specific application, but at least, let us mention some popular examples:
\begin{example}
\begin{rmklist}
\item
	\define{Square loss.} The square loss 
	\begin{equation}
	\losssq\colon \R \times \Y = \R \times \R \to \R, \ (v,y) \mapsto \tfrac{1}{2} (v-y)^2
	\end{equation}
	has been extensively discussed in the previous parts and has served as the ``prototype'' example for our results. We have seen that the definitions of the model parameters $\scalfac$, $\modeldev$, and $\modeldeveta$ simplify significantly (cf. \eqref{eq:intro:modelparam}), and since $(\losssq)'' \equiv 1$, the RSC is immediately implied by Theorem~\ref{thm:results:rsc} for arbitrarily large values of $M$ and $N$.
\item
	\define{Logistic loss.} When working with a classification model such as in Example~\ref{ex:consequences:model}\ref{ex:consequences:model:onebit}, the square loss might suffer from the drawback that it also penalizes correctly classified samples. In this case, it could be more beneficial to use the \define{logistic loss}\footnote{Note that we disregard that $\sign(0) = 0$ here, since this case occurs with probability zero anyway.}
	\begin{equation}\label{eq:consequences:loss:logloss}
		\loss^{\text{log}}\colon \R \times \Y = \R \times \{-1,+1\} \to \R, \ (v,y) \mapsto - y \cdot v + \log(1+\exp(- y\cdot v)).
	\end{equation}
	One easily checks that \ref{assump:estimator:difflip}--\ref{thm:results:rsc:strongconvexity} hold true so that $\loss^{\text{log}}$ is admissible for our framework. Note that, statistically, the resulting estimator $(P_{\loss^{\text{log}},\sset})$ corresponds to minimizing the negative log-likelihood of a logistic regression model, combined with a structural constraint.
\item
	\define{Generalized linear models.} Motivated by the previous example, one could also consider maximum likelihood estimators of generalized linear models. The related loss function then takes the form $\loss(v,y) = -y \cdot v + \Phi(v)$, where $\Phi\colon \R \to \R$ denotes the \define{link function} of the model; see \cite[Sec.~4.4]{negahban2012unified} for more details.
	\qeddiamond
\end{rmklist}
\end{example}

It has been repeatedly pointed out that the model parameters $\scalfac$, $\modeldev$, and $\modeldeveta$ are measuring how well the loss function and the (non-)linear model interact with each other. There are in fact cases where this interplay is so ``bad'' that the defining equation \eqref{eq:results:recovery:modelparam:scalfac} for $\scalfac$ cannot be solved, and our theorems are not applicable anymore. The following example shows such an ``incompatible'' pair of $\loss$ and $\fobs$.
\begin{example}
Suppose that we are trying to fit noiseless $1$-bit observations $y_i = \sign(\sp{\a_i}{\grtr})$ with the logistic loss $\loss = \loss^{\text{log}}$ defined in \eqref{eq:consequences:loss:logloss}. Then, for any $\lambda > 0$, we have
\begin{equation}
\lossemp[\y](\lambda\grtr) = \tfrac{1}{m}\sum_{i = 1}^m \big( - \lambda \abs{\sp{\a_i}{\grtr}} + \log(1+\exp(- \lambda \abs{\sp{\a_i}{\grtr}})) \big),
\end{equation}
implying that the empirical loss can become arbitrarily small. Therefore, the estimator \eqref{eq:intro:estimator} is only guaranteed to have a minimizer $\solu$ if $\sset$ is compact. But even then we cannot hope for a recovery of $\grtr$, although $\solu$ probably provides a perfect classification. This example particularly illustrates that approximating the ground-truth signal $\grtr$ is a much more difficult task than just finding a good classifier for our model.
And interestingly, it also defies conventional
wisdom of using logistic regression as the standard method for a binary output scheme. \qeddiamond
\end{example}

\subsection{Descent Cones and the Convex Geometry of Theorem \ref{thm:results:recovery}}
\label{subsec:consequences:convexgeo}

In this part, we shall briefly discuss the underlying geometric ideas of our approach.
For the sake of simplicity, let us assume that we are working in the noiseless linear regime, i.e., $y_i = \sp{\a_i}{\grtr}$, or in short $\y = \yadv = \A \grtr$. We would like to perform the $\sset$-Lasso \eqref{eq:intro:klasso} for some convex signal set $\sset \subset \R^n$ with $\grtr \in \sset$. Then, $\grtr$ is obviously a minimizer of \eqref{eq:intro:klasso} and we might ask when it is unique, meaning that the recovery is exact. It is not hard to see (cf. \cite[Sec.~1.9.1]{vershynin2014estimation}) that this is the case if and only if
\begin{equation}\label{eq:consequences:convexgeo:exactrecov}
	(\sset - \grtr) \intersec \ker\A = \{ \vec{0} \}.
\end{equation}
Since $\sset$ is convex, this condition can be easily rewritten in terms of the so-called \define{descent cone} of $\sset$ at $\grtr$, given by\footnote{Note that $\cone{\sset}{\grtr}$ might not be a closed set even when $\sset$ is closed. As an example, consider $\sset = \ball[2][n]$ and $\grtr = (1, 0, \dots, 0)$.}
\begin{equation}
	\cone{\sset}{\grtr} := \{\lambda (\x - \grtr) \suchthat \x \in \sset, \ \lambda \geq 0 \}.
\end{equation}
Indeed, \eqref{eq:consequences:convexgeo:exactrecov} is equivalent to
\begin{equation}
	\cone{\sset}{\grtr} \intersec \ker\A = \{ \vec{0} \}.
\end{equation}
This observation originates from \cite[Prop.~2.1]{chandrasekaran2012geometry}; see Figure~\ref{fig:consequences:convexgeo:descentcone} for a visualization.
Since the matrix $\A \in \R^{m \times n}$ is i.i.d. Gaussian, $\ker\A$ can be identified with an $(n-m)$-dimensional random subspace drawn from the Grassmanian $G(n,n-m)$. The problem of exact recovery can be therefore reduced to the question of when a random subspace of codimension $m$ trivially intersects the descent cone $\cone{\sset}{\grtr}$. Heuristically, the latter criterion should be met with high probability if the cone $\cone{\sset}{\grtr}$ is relatively ``narrow.'' This important geometrical insight was formalized in \cite{chandrasekaran2012geometry} by Gordon's escape, leading to a sufficient condition of the form
\begin{equation}\label{eq:consequences:convexgeo:escape}
m \gtrsim \effdim[t]{\cone{\sset}{\grtr}} = \meanwidth[1]{\cone{\sset}{\grtr}}^2,
\end{equation}
where we have used that the effective dimension of cones is independent from the scale $t$. This particularly underpins our intuition that the problem becomes ``easier'' to solve as the codimension of the random subspace grows or the cone becomes more narrow (measured in terms of the mean width).
Remarkably, \eqref{eq:consequences:convexgeo:escape} precisely coincides with our findings of Example~\ref{ex:consequences:model}\ref{ex:consequences:model:linear}.

\begin{figure}
	\centering
	\begin{subfigure}[t]{0.35\textwidth}
		\centering
		\includegraphics[width=\textwidth]{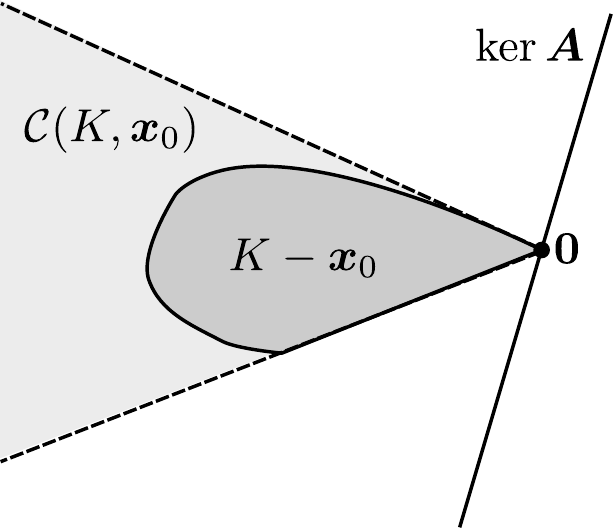}
		\caption{}
		\label{fig:consequences:convexgeo:descentcone}
	\end{subfigure}%
	\qquad\qquad
	\begin{subfigure}[t]{0.35\textwidth}
		\centering
		\includegraphics[width=\textwidth]{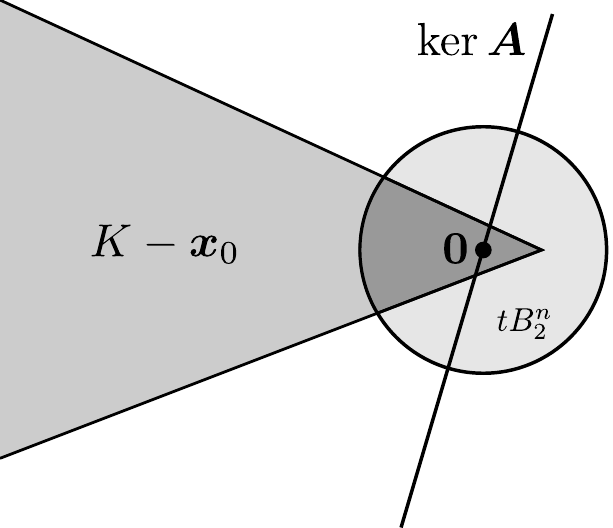}
		\caption{}
		\label{fig:consequences:convexgeo:localmeanwidth}
	\end{subfigure}%
	\caption{\subref{fig:consequences:convexgeo:descentcone} The convex geometry of exact recovery. Note that everything is ``centered'' around $\vec{0}$, since we have subtracted $\grtr$. \subref{fig:consequences:convexgeo:localmeanwidth} Situation where $\vec{0}$ is not exactly located at the apex of $\sset - \grtr$. The effective dimension is computed over the intersection of $\sset - \grtr$ and $t\ball[2][n]$. For simplicity, we visualize the signal set by a cone, but in general, it could be any convex set.}
	\label{fig:consequences:convexgeo}
\end{figure}

However, it could happen that $\grtr$ does not lie exactly at the boundary of $\sset$, or even worse, the signal set has ``smoothed vertices''---think of $\sset = \sqrt{s}\ball[1][n] \intersec \ball[2][n]$ for instance. We would essentially have 
\begin{equation}
	\effdim[1]{\cone{\sset}{\grtr}} = \meanwidth[1]{\cone{\sset}{\grtr}}^2 \asymp n
\end{equation}
in these cases,\footnote{We assume that $\sset$ has full algebraic dimension here.} and the condition \eqref{eq:consequences:convexgeo:escape} becomes meaningless. Fortunately, it has turned out in \cite{plan2014highdim,plan2015lasso} that the \emph{local} mean width allows us to overcome this drawback. In order to get a basic understanding for this approach, let us consider Figure~\ref{fig:consequences:convexgeo:localmeanwidth}, where $\grtr$ is relatively closely located to the apex of $\sset$. 
Again, the intersection between $\sset - \grtr$ and $\ker\A$ could be very small if $\sset$ is narrow, which would still lead to an accurate (maybe not exact) estimate of $\grtr$ by \eqref{eq:intro:klasso}.
But how to measure the size of $\sset$ in this situation? The key idea is that, if the scale $t > 0$ is sufficiently large, the value of
\begin{equation}
\effdim[t]{\sset - \grtr} = \Big(\tfrac{1}{t} \cdot \mean[(\sup_{\x \in (\sset - \grtr)\intersec t\ball[2][n]} \sp{\gaussian}{\x}]\Big)^2
\end{equation}
is approximately equal to the (local) effective dimension in the case where $\grtr$ lies \emph{exactly} at the apex of $\sset$. For that reason, one may apply the same argumentation as above, but without the necessity of forming a descent cone (which could be the entire $\R^n$).
According to the setup of Theorem~\ref{thm:results:recovery}, the scaling parameter $t$ can be also interpreted as error accuracy, since it essentially bounds the size of the set of potential minimizers $\sset\intersec(\ker\A + \grtr)$.
On the other hand, if the desired ``resolution'' is too high, i.e., $t \to 0$, the scaled ball $t\ball[2][n]$ might be completely contained in the interior of $\sset - \grtr$, so that we may end up with $\effdim[t]{\sset - \grtr} \asymp n$ at some point.
This transition behavior opens up a new (geometric) perspective of the recovery guarantee in Theorem~\ref{thm:results:recovery} and particularly indicates that the local mean width is indeed an appropriate refinement of the classical ``conic mean width.''

\section{Conclusion and Outlook}
\label{sec:conclusion}

Our main results, Theorem~\ref{thm:results:recovery}, Theorem~\ref{thm:results:rsc}, and Theorem~\ref{thm:results:aniso}, have shown that high-dimensional estimation is possible under fairly general assumptions, including non-linear observations with adversarial noise, strictly convex loss functions as well as arbitrary convex signal sets. 
In this setup, perhaps the most remarkable conclusion was that these guarantees can be already achieved with standard estimators that do not rely on any knowledge of the non-linear and noisy output rule.
We have seen that the actual capability of recovering a signal is essentially based on the following key concepts, originating from the works of \cite{negahban2009unified,negahban2012unified,plan2015lasso,plan2014highdim}, respectively:
\begin{thmlist}
\item
	The \emph{(local) mean width}, capturing the complexity of the signal class.
\item
	\emph{Restricted strong convexity} (\emph{RSC}), allowing us to turn a statement on consistent prediction into a statement on signal approximation.
\item
	The \emph{model parameters} $\scalfac$, $\modeldev$, $\modeldeveta$, quantifying the uncertainty caused by a (random) non-linear perturbation of linear measurements.
\end{thmlist}
Let us conclude this section with a brief listing of potential extensions of our framework and some open questions, which could be considered for future work:
\begin{listing}
\item
	\emph{Non-Gaussian measurements.} The works of Tropp \cite{tropp2014convex} and Mendelson \cite{mendelson2014learning,mendelson2014learninggeneral} go beyond the assumption of Gaussian random vectors. 
	Therefore, one might wonder whether our setting could be also extended to more general distributions, e.g., \define{sub-Gaussians}.
	Some arguments of our proofs can be easily adapted (see also \cite[Prop.~2.5.1, Thm.~2.6.3]{tropp2014convex}), but there are other steps whose generalization seems to be more involving, for example, a non-linear observation model with sub-Gaussian measurements (cf. \cite{ai2014onebitsubgauss}).
\item
	\emph{More general convex loss functions.} A popular example which is not included in our setting are \define{support vector machines} (SVMs) based on the \define{hinge loss} $\loss^\text{hgn}(v,y) := \max\{0, 1 - v \cdot y\}$. Although $\loss^\text{hgn}$ obviously does not satisfy RSC, some recovery guarantees for sparse vectors can be still proven; see \cite{kolleck2015l1svm} for instance. Thus, we might ask for a generalization of the RSC-condition from Definition~\ref{def:results:rsc}.
	One promising strategy could be a replacement of the Euclidean norm in \eqref{eq:results:rsc:rsc} by a different measure of distance, possibly adapted to the used loss function.
	
	Apart from that, the sufficient conditions for RSC in Theorem~\ref{thm:results:rsc} might be further relaxed. 
	Especially some restrictions on the second input variable of $(v,y) \mapsto \loss(v,y)$ could be dropped, since all regularity assumptions are imposed with respect to the first variable.
	Furthermore, it should be also possible to allow for piece-wise smooth functions: Such an improvement could be achieved by first using a smooth approximation of the loss in the proof and taking a uniform limit afterwards.
\item
	\emph{``Optimal'' loss functions.} In Section~\ref{sec:consequences}, we have discussed a few common examples of loss functions, but practitioners often have to face the following challenge: Supposed that we have some approximate knowledge of the true observation scheme, what is an (almost) optimal choice of $\loss$? 
	A particular difficulty here is to introduce a reasonable notion of optimality.
	As a first step, one could try to construct a loss function such that $\scalfac$ becomes as large as possible compared to $\modeldev^2$ and $\modeldeveta^2$. This would ensure a good \emph{signal-to-noise ratio}, since the estimator \eqref{eq:intro:estimator} measures residuals in terms of $\loss$ (see also the discussion part of Theorem~\ref{thm:results:recovery}).
	However, it is unclear whether this strategy really leads to the ``best possible'' loss function.
	
	Another important issue concerns the robustness against model inaccuracies. For instance, one might ask what would happen if the output actually obeys a probit model while the logistic loss is used for estimation. It would be interesting to study how sensitive our estimator is to such a type of mismatch. Moreover, one could develop some general rules-of-thumb indicating when $\fobs$ and $\loss$ are particularly ``compatible.''
\item
	Instead of a constrained estimator, we could consider a \emph{regularized} version of it, such as in \cite{negahban2009unified,negahban2012unified}. This might be more appealing for computational purposes, and we strongly believe that similar recovery results can be shown if the corresponding regularization parameter is appropriately chosen.
\item
	Finally, one could think of extending the single-index model \ref{assump:model:nonlinmod} to a \define{multi-index model} of the form
	\begin{equation}
		y_i = \fobs(\sp{\a_i}{\grtr^{(1)}}, \sp{\a_i}{\grtr^{(2)}}, \dots, \sp{\a_i}{\grtr^{(N)}}), \quad i = 1,\dots, m.
	\end{equation}
	The major goal is then to find natural conditions on the output function $\fobs \colon \R^N \to \Y$ for which efficient recovery of the unknown signals $\grtr^{(1)}, \dots, \grtr^{(N)} \in \R^n$ is feasible.
	This challenge is also closely related to the theory of \define{generalized ridge functions}; see \cite{cohen2012ridge,fornasier2012ridge} for example.
\end{listing}

\section{Proofs of the Main Results}
\label{sec:proof}

\subsection{Proof of Theorem \ref{thm:intro:recovery}}
\label{subsec:proof:introrecovery}

\begin{proof}[Proof of Theorem~\ref{thm:intro:recovery}]
We would like to apply Theorem~\ref{thm:results:aniso} for $\Covmatr = \I$ and
\begin{equation}
t = C' \Big( \Big(\frac{\meanwidth{\sset - \grtrmu}^2}{m}\Big)^{1/4} + \advdev \Big),
\end{equation}
where $C' > 0$ is a constant which is chosen later. First, note that the assumption \ref{thm:results:rsc:strongconvexity} is ensured by Corollary~\ref{cor:results:rsc:strictlyconvex}. Next, one observes that
\begin{equation}\label{eq:proof:introrecovery:effdim}
\effdim[t]{\sset - \grtrmu} = \frac{1}{t^2} \cdot \meanwidth[t]{\sset - \grtrmu}^2 \stackrel{\eqref{eq:results:signal:localglobal}}{\leq} \frac{1}{t^2} \cdot \meanwidth{\sset - \grtrmu}^2 \leq \frac{1}{(C')^2} \cdot \sqrt{m} \cdot \meanwidth{\sset - \grtrmu}.
\end{equation}
Together with the assumption \eqref{eq:intro:recovery:meascount} for some $C \geq C_1^2/(C')^4 $, we obtain $m \geq C_1 \effdim[t]{\sset - \grtrmu}$.\footnote{In this proof, $C_1, C_2 > 0$ denote the constants from the statement of Theorem~\ref{thm:results:aniso}.} Finally, we can choose $C' = C'' \cdot \max\{1, \modeldev, \modeldeveta\}$, with some $C'' > 0$ depending on $C_2$, such that
\begin{align}
	C_2 \cdot \accuracy &= C_2 \Big(\frac{\modeldev \cdot \sqrt{\effdim[t]{\sset - \grtrmu}} + \modeldeveta}{\sqrt{m}} + \advdev\Big) < (C')^2 \cdot \frac{\sqrt{\effdim[t]{\sset - \grtrmu}}}{\sqrt{m}} + C' \cdot \advdev \\
	&\stackrel{\eqref{eq:proof:introrecovery:effdim}}{\leq} C' \Big(\frac{\meanwidth{\sset - \grtrmu}^2}{m}\Big)^{1/4} + C' \cdot \advdev  = t.
\end{align}
The claim is now implied by Theorem~\ref{thm:results:aniso}. Note that, since $C'' \leq C'' \cdot \max\{1, \modeldev, \modeldeveta\} = C'$, we can always enlarge $C''$ and $C$ such that $C'' = C$.
\end{proof}

\subsection{Proof of Theorem \ref{thm:results:recovery}}
\label{subsec:proof:recovery}

Throughout this subsection, we shall assume that the assumptions of Theorem~\ref{thm:results:recovery} hold true.
At first, let us recall that the generalized estimator \eqref{eq:intro:estimator} can be written in terms of the empirical loss function:
\begin{equation}
	\min_{\x \in \R^n} \lossemp[\yadv](\x) \quad \text{subject to $\x \in \sset$,} \label{eq:proof:estimator-emp}\tag{$P_{\loss, \sset}^\text{emp}$}
\end{equation}
where $\lossemp[\yadv](\x) = \tfrac{1}{m} \sum_{i = 1}^m \loss(\sp{\a_i}{\x}, \tilde{y}_i)$. The basis for our analysis is the first order error approximation of $\lossemp[\yadv](\x)$ at $\grtrmu$ (see also \eqref{eq:results:recovery:taylor}):
\begin{equation}\label{eq:proof:recovery:taylor}
	\losstaylor{\x}{\grtrmu}{\yadv} = \lossemp[\yadv](\x) - \lossemp[\yadv](\grtrmu) - \sp{\gradient\lossemp[\yadv](\grtrmu)}{\x - \grtrmu}.
\end{equation}
By the assumption of RSC, we may estimate $\losstaylor{\x}{\grtrmu}{\yadv}$ from below by $\lnorm{\x - \grtrmu}^2$ for all $\x \in \sset \intersec (t \S^{n-1} + \grtrmu)$. 
Hence, in order to control this distance, we need to find an appropriate upper bound for the right-hand side of \eqref{eq:proof:recovery:taylor}. For this, let us introduce two vectors $\z = (z_1, \dots, z_m)$ and $\zadv = (\tilde{z}_1, \dots, \tilde{z}_m)$ with
\begin{align}
	z_i &:= \loss'(\sp{\a_i}{\grtrmu}, y_i), \quad i \in [m],\\
	\tilde{z}_i &:= \loss'(\sp{\a_i}{\grtrmu}, \tilde{y}_i), \quad i \in [m].
\end{align}
\begin{remark}[Key idea of the proof, cf. {\cite[Sec.~I.E]{plan2015lasso}}]\label{rmk:proof:recover:intuitionz}
The purpose of these two ``residual'' vectors becomes clearer when considering the case of the square loss $\loss = \losssq$. Here, we simply have
\begin{equation}
\y = \scalfac \A \grtr - \z \quad \text{and} \quad \yadv = \scalfac \A \grtr - \zadv,
\end{equation}
meaning that $\z$ (and $\zadv$) describe the deviation of the non-linear model (with adversarial noise) from the underlying linear measurement process. Although $\z$ might not have mean zero and could depend on $\A$, it can be still regarded as a special type of ``noise.'' In fact, the scaling parameter $\scalfac$ was precisely chosen such that $\z$ and $\A$ are uncorrelated, i.e., $\mean[\A^\T \z] = \vec{0}$ (see \cite[Eq.~(IV.1)]{plan2015lasso}).
Later on (in Lemma~\ref{lem:proof:recovery:residual}), we will see that this observation is a crucial step towards proving our error bounds. \qeddiamond
\end{remark}
Next, we ``isolate'' the adversarial noise term in \eqref{eq:proof:recovery:taylor}:
\begin{align}
	\losstaylor{\x}{\grtrmu}{\yadv} ={} & \lossemp[\yadv](\x) - \lossemp[\yadv](\grtrmu) - \sp{\gradient\lossemp[\yadv](\grtrmu)}{\x - \grtrmu} \\
	={} & \lossemp[\yadv](\x) - \lossemp[\yadv](\grtrmu) - \tfrac{1}{m} \sp{ \A^\T \zadv}{\x - \grtrmu}\\
	={} & \lossemp[\yadv](\x) - \lossemp[\yadv](\grtrmu) - \tfrac{1}{m} \sp{ \A^\T \z}{\x - \grtrmu} + \tfrac{1}{m} \sp{ \A^\T (\z - \zadv)}{\x - \grtrmu} \\
	={} & \lossemp[\yadv](\x) - \lossemp[\yadv](\grtrmu) + \underbrace{\tfrac{1}{m} \sp{ \A^\T \z}{\grtrmu-\x}}_{=:T_1} + \underbrace{\sp{\z - \zadv}{\tfrac{1}{m}\A(\x - \grtrmu)}}_{=: T_2}. \label{eq:proof:recovery:basiceq}
\end{align}
This \emph{basic equation} forms the starting point of the proof of Theorem~\ref{thm:results:recovery}.
In order to \emph{uniformly} bound the terms $T_1$ and $T_2$, we provide two lemmas, which are proven at the end of this subsection:

\begin{lemma}\label{lem:proof:recovery:residual}
Let $L \subset t \ball[2][n]$ be any subset. Then
\begin{equation}
	\mean[\sup_{\h \in L} \sp{\A^\T \z}{\h}] \leq (\modeldev\cdot \meanwidth{L} + t \cdot \modeldeveta) \cdot \sqrt{m}.
\end{equation}
\end{lemma}

\begin{lemma}\label{lem:proof:recovery:noise}
Let $L \subset t \ball[2][n]$ be any subset. Then, there exists a numerical constant $C > 0$ such that
\begin{equation}
	\sup_{\h \in L} \tfrac{1}{\sqrt{m}} \lnorm{\A\h} \leq C \cdot \Big(t + \frac{\meanwidth{L}}{\sqrt{m}}\Big)
\end{equation}
with probability at least $1 - \exp(-m)$.
\end{lemma}

Now, we are ready to prove Theorem~\ref{thm:results:recovery}:
\begin{proof}[Proof of Theorem \ref{thm:results:recovery}]
\begin{proofsteps}
\item
At first, let us fix a vector $\x \in \sset \intersec (t \S^{n-1} + \grtrmu)$, which particularly means that $\lnorm{\x - \grtrmu} = t$.
Note that w.l.o.g., we can assume that $\sset \intersec (t \S^{n-1} + \grtrmu) \neq \emptyset$ because otherwise, the convexity of $\sset$ would automatically imply that $\lnorm{\x - \grtrmu} \leq t$ for  all $\x \in \sset$.
In the following two steps, we establish bounds for $T_1$ and $T_2$:

\item
\emph{Bounding $T_1$:}
For this, we apply Lemma~\ref{lem:proof:recovery:residual} with $L = (\grtrmu - \sset) \intersec t \ball[2][n]$. 
Since $\grtrmu - \x \in L$, Markov's inequality yields
\begin{equation}
T_1 = \tfrac{1}{m} \sp{ \A^\T \z}{\grtrmu - \x} \leq C_1 \cdot \frac{\modeldev\cdot \meanwidth{L} + t \cdot \modeldeveta}{\sqrt{m}} = C_1 \cdot \frac{\modeldev\cdot \meanwidth[t]{\sset - \grtrmu} + t \cdot \modeldeveta}{\sqrt{m}} \label{eq:proof:recovery:boundT1}
\end{equation}
with high probability, where the constant $C_1 > 0$ only depends on the probability of success.
\item
\emph{Bounding $T_2$:} Let us recall the model assumption \ref{assump:model:advnoise} and the definitions of $\z$ and $\zadv$. Using the Lipschitz continuity of $\loss'$ (cf. \ref{assump:estimator:difflip}), we conclude
\begin{align}
\lnorm{\z - \zadv}^2 &= \sum_{i = 1}^m \abs{\loss'(\sp{\a_i}{\grtrmu}, y_i) - \loss'(\sp{\a_i}{\grtrmu}, \tilde{y}_i)}^2 \\
&\leq C_{\loss'}^2 \cdot \sum_{i = 1}^m \abs{y_i - \tilde{y}_i}^2 = C_{\loss'}^2 \cdot \lnorm{\y - \yadv}^2.
\end{align}
Thus, the Cauchy-Schwarz inequality implies
\begin{align}
T_2 &= \sp{\z - \zadv}{\tfrac{1}{m}\A(\x - \grtrmu)} \leq \tfrac{1}{\sqrt{m}} \lnorm{\z - \zadv} \cdot \tfrac{1}{\sqrt{m}} \lnorm{\A(\x - \grtrmu)} \\
&\leq C_{\loss'} \cdot \underbrace{\tfrac{1}{\sqrt{m}}\lnorm{\y - \yadv}}_{= \advdev} \cdot \tfrac{1}{\sqrt{m}} \lnorm{\A(\x - \grtrmu)}.
\end{align}
Finally, we can invoke the bound of Lemma~\ref{lem:proof:recovery:noise} with $L = (\sset - \grtrmu) \intersec t \ball[2][n]$, which leads to
\begin{equation}
T_2 \leq C_{\loss'} \cdot \advdev \cdot C \cdot \Big(t + \underbrace{\frac{\meanwidth[t]{K - \grtrmu}}{\sqrt{m}}}_{\stackrel{\eqref{eq:results:recovery:measurements}}{\leq} t}\Big) \leq C_2 \cdot t \cdot \advdev \label{eq:proof:recovery:boundT2}
\end{equation}
with high probability, where the constant $C_2 > 0$ only depends on (the Lipschitz constant of) $\loss'$. 
\item
Now, let us assume that the events of \eqref{eq:proof:recovery:boundT1} and \eqref{eq:proof:recovery:boundT2} have indeed occurred. Together with the RSC at scale $t$ (with RSC-constant $C_3 > 0$), we end up with the following  estimate for the basic equation \eqref{eq:proof:recovery:basiceq}:
\begin{align}
	C_3 \cdot t^2 &= C_3 \lnorm{\x - \grtrmu}^2 \leq \losstaylor{\x}{\grtrmu}{\yadv} = \lossemp[\yadv](\x) - \lossemp[\yadv](\grtrmu) + T_1 + T_2 \\
	&\leq \lossemp[\yadv](\x) - \lossemp[\yadv](\grtrmu) + C_1 \cdot \frac{\modeldev\cdot \meanwidth[t]{K - \grtrmu} + t \cdot \modeldeveta}{\sqrt{m}} + C_2 \cdot t \cdot \advdev \\
	& \leq \lossemp[\yadv](\x) - \lossemp[\yadv](\grtrmu) + \max\{C_1, C_2\} \cdot t \cdot \underbrace{\Big( \frac{\modeldev\cdot \sqrt{\effdim[t]{K - \grtrmu}} + \modeldeveta}{\sqrt{m}} + \advdev \Big)}_{= \accuracy}.
\end{align}
By choosing $C = \max\{C_1, C_2\} / C_3$ in \eqref{eq:results:recovery:accuracy}, we finally obtain
\begin{equation}
0 < t (C_3 \cdot t - \max\{C_1, C_2\}\cdot \accuracy) \leq \lossemp[\yadv](\x) - \lossemp[\yadv](\grtrmu) =: \score(\x).
\end{equation}
This particularly implies that $\score(\x) > 0$ for all $\x \in \sset \intersec (t \S^{n-1} + \grtrmu)$. On the other hand, any minimizer $\solu \in \sset$ of \eqref{eq:proof:estimator-emp} satisfies $\score(\solu) \leq 0$. Hence, if we would have $\lnorm{\solu - \grtrmu} > t$, there would exist (by convexity of $\sset$) some $\x' \in \sset \intersec (t \S^{n-1} + \grtrmu)$ such that $\solu \in \{ \grtrmu + \lambda \x' \suchthat \lambda > t \}$. But this already contradicts the fact that $\x \mapsto \score(\x)$ is a convex functional, since it holds $\score(\solu) \leq 0$, $\score(\x') > 0$, and $\score(\grtrmu) = 0$. Consequently, we have $\lnorm{\solu - \grtrmu} \leq t$, which proves the claim. \qedhere
\end{proofsteps}
\end{proof}

It remains to prove Lemma~\ref{lem:proof:recovery:residual} and Lemma~\ref{lem:proof:recovery:noise}:
\begin{proof}[Proof of Lemma \ref{lem:proof:recovery:residual}]
This proof is literally adapted from \cite[Lem.~4.3]{plan2015lasso}, but to keep this work self-contained, we provide the details here. As already indicated in Remark~\ref{rmk:proof:recover:intuitionz}, we would like to ``decouple'' $\z$ and $\A$ as much as possible. For that purpose, we consider the projection $\proj := \grtr \grtr^\T$ onto the span of the ground-truth signal $\grtr$ as well as the projection $\orthcompl{\proj} = \I - \grtr \grtr^\T$ onto its orthogonal complement $\orthcompl{\{\grtr\}}$. Note that these projections are orthogonal, since $\grtr \in \S^{n-1}$ by assumption. Moreover, we define the functional $\normpol{\x}{L} := \sup_{\h \in L} \sp{\x}{\h}$ for $\x \in \R^n$.\footnote{If $L$ is a symmetric convex body, inducing a norm $\norm{\cdot}_L$, then $\normpol{\cdot}{L}$ is precisely the associated dual norm.}

Using that $\orthcompl{\proj} + \proj = \I$, we obtain
\begin{equation}
\mean[\sup_{\h \in L} \sp{\A^\T \z}{\h}] = \mean\normpol{\A^\T \z}{L} \leq \underbrace{\mean\normpol{\orthcompl{\proj}\A^\T \z}{L}}_{=: T_1'} + \underbrace{\mean\normpol{\proj\A^\T \z}{L}}_{=: T_2'}.
\end{equation}
Let us first estimate $T_1'$. Since $\A$ is a Gaussian matrix, $\orthcompl{\proj}\A^\T$ and $\proj\A^\T$ are independent of each other. Observing that the columns of $\proj\A^\T$ are given by $\sp{\a_i}{\grtr}\grtr$ and that 
\begin{equation}
z_i = \loss'(\sp{\a_i}{\grtrmu}, y_i) = \loss'(\scalfac\sp{\a_i}{\grtr}, \fobs(\sp{\a_i}{\grtr})),
\end{equation}
it turns out that $\orthcompl{\proj}\A^\T$ is also independent of $\z$ (entry-wise). Thus, introducing an independent copy $\Aind$ of $\A$ (also independent of $\z$), the random vectors $\orthcompl{\proj}\A^\T\z$ and $\orthcompl{\proj}\Aind^\T\z$ have the same probability distribution. Therefore,
\begin{equation}\label{eq:proof:recovery:residual:T1}
T_1' = \mean\normpol{\orthcompl{\proj}\A^\T \z}{L} = \mean\normpol{\orthcompl{\proj}\Aind^\T \z}{L} = \mean\normpol{(\orthcompl{\proj}\Aind^\T + \underbrace{\mean[\proj\Aind^\T]}_{= \vec{0}}) \z}{L}. 
\end{equation}
Now, by Jensen's inequality with respect to $\proj\Aind^\T$,\footnote{Note that the outer expectation in \eqref{eq:proof:recovery:residual:T1} is taken over $\A$ and $\orthcompl{\proj}\Aind^\T$, whereas the inner expectation is taken over $\proj\Aind^\T$.} \eqref{eq:proof:recovery:residual:T1} can be bounded by
\begin{equation}
\mean\normpol{(\orthcompl{\proj}\Aind^\T + \proj\Aind^\T) \z}{L} = \mean\normpol{\Aind^\T\z}{L}.
\end{equation}
Conditioning on $\z$, we observe that $\Aind^\T\z$ has the same distribution as $\lnorm{\z} \cdot \gaussian$, where $\gaussian \distributed \Normdistr{\vec{0}}{\I}$ is independent of $\z$. So, we finally obtain
\begin{align}
T_1' &\leq \mean\normpol{\Aind^\T\z}{L} = \mean[\lnorm{\z} \cdot \sup_{\h \in L} \sp{\gaussian}{\h}] = \mean[\lnorm{\z}] \cdot \mean[\sup_{\h \in L} \sp{\gaussian}{\h}]\\
&\leq \sqrt{\mean[\lnorm{\z}^2]} \cdot \meanwidth{L} = \sqrt{m} \cdot \modeldev \cdot \meanwidth{L},
\end{align}
where in the last step, we have used the definition of $\modeldev$ (cf. \eqref{eq:results:recovery:modelparam:modeldev}) as well as $\sp{\a_i}{\grtr} \distributed \Normdistr{0}{1}$.

Next, we shall bound the term $T_2'$. For this purpose, we introduce new random variables $\xi_i := z_i \sp{\a_i}{\grtr}$ for $i \in [m]$ and observe that
\begin{equation}
\proj\A^\T \z = \sum_{i = 1}^m z_i \sp{\a_i}{\grtr} \grtr = \Big(\sum_{i = 1}^m \xi_i\Big) \cdot \grtr.
\end{equation}
By the Cauchy-Schwarz inequality and $L \subset t \ball[2][n]$, we therefore obtain
\begin{align}
T_2' &= \mean\normpol{\proj\A^\T \z}{L} = \mean[\sup_{\h \in L} \sp{(\textstyle\sum_{i = 1}^m \xi_i) \cdot \grtr}{\h}] \\
&\leq \mean{}[\sup_{\h \in L} (\abs{\textstyle\sum_{i = 1}^m \xi_i} \cdot \underbrace{\lnorm{\grtr}}_{= 1} \cdot \underbrace{\lnorm{\h}}_{\leq t})] \leq t \cdot \mean\abs{\textstyle\sum_{i = 1}^m \xi_i}.
\end{align}
Finally, we apply the definitions of $\scalfac$ and $\modeldeveta$ in \eqref{eq:results:recovery:modelparam} to observe that\footnote{Note that this is actually the only place in the proof of Theorem~\ref{thm:results:recovery} where we apply the definition of $\scalfac$. In fact, it was precisely chosen such that $\xi_1, \dots, \xi_m$ have mean zero.}
\begin{align}
\mean[\xi_i] &= \mean[\loss'(\scalfac\sp{\a_i}{\grtr}, \fobs(\sp{\a_i}{\grtr})) \cdot \sp{\a_i}{\grtr}] = 0, \\
\mean[\xi_i^2] &= \mean[\loss'(\scalfac\sp{\a_i}{\grtr}, \fobs(\sp{\a_i}{\grtr}))^2 \cdot \sp{\a_i}{\grtr}^2] = \modeldeveta^2,
\end{align}
where we have again used that $\sp{\a_i}{\grtr} \distributed \Normdistr{0}{1}$. Together with the (pairwise) independence of $\xi_1, \dots, \xi_m$, this implies
\begin{equation}
T_2' \leq t \cdot \mean\abs{\textstyle\sum_{i = 1}^m \xi_i} \leq t \cdot \sqrt{\mean[(\textstyle \sum_{i = 1}^m \xi_i)^2]} = t \cdot \sqrt{\textstyle \sum_{i = 1}^m \mean[\xi_i^2]} = t \cdot \sqrt{m} \cdot \modeldeveta.
\end{equation}
The proof is now complete. \qedhere
\end{proof}

\begin{proof}[Proof of Lemma \ref{lem:proof:recovery:noise}]
The key ingredient of our proof is a uniform concentration inequality for \mbox{(sub-)Gaussian} random matrices, which was proven in \cite[Thm.~1.3]{liaw2016randommat}:
\begin{equation}\label{eq:proof:recovery:noise:concentration}
\prob[\sup_{\h \in L} \abs{\tfrac{1}{\sqrt{m}} \lnorm{\A\h} - \lnorm{\h}} \leq C' \cdot \tfrac{\meanwidth{L} + u \cdot \sup_{\h'\in L}\lnorm{\h'}}{\sqrt{m}}] \geq 1 - \exp(-u^2)
\end{equation}
for all $u \geq 0$ and a numerical constant $C' > 0$.

Setting $u := \sqrt{m}$ and using that $L \subset t \ball[2][n]$, we obtain the following bound from \eqref{eq:proof:recovery:noise:concentration}:
\begin{align}
1 - \exp(-m) \leq{} &\prob{}[\sup_{\h \in L} \underbrace{\abs{\tfrac{1}{\sqrt{m}} \lnorm{\A\h} - \lnorm{\h}}}_{\geq \tfrac{1}{\sqrt{m}} \lnorm{\A\h} - t} \leq C' \cdot ( \tfrac{\meanwidth{L}}{\sqrt{m}} + \underbrace{\sup_{\h'\in L}\lnorm{\h'}}_{\leq t})] \\
\leq{}& \prob{}[\sup_{\h \in L} \tfrac{1}{\sqrt{m}} \lnorm{\A\h} - t \leq C' \cdot ( \tfrac{\meanwidth{L}}{\sqrt{m}} + t)] \\
={}& \prob{}[\sup_{\h \in L} \tfrac{1}{\sqrt{m}} \lnorm{\A\h} \leq C' \cdot \tfrac{\meanwidth{L}}{\sqrt{m}} + \underbrace{(C' + 1)}_{=: C} t] \\
\leq{}& \prob{}[\sup_{\h \in L} \tfrac{1}{\sqrt{m}} \lnorm{\A\h} \leq C \cdot (\tfrac{\meanwidth{L}}{\sqrt{m}} + t)],
\end{align}
which is precisely the claim of Lemma~\ref{lem:proof:recovery:noise}.
\end{proof}

\subsection{Proof of Theorem \ref{thm:results:rsc}}
\label{subsec:proof:rsc}

Throughout this subsection, we shall assume that the assumptions of Theorem~\ref{thm:results:rsc} are indeed satisfied.
Since $\loss$ is twice continuously differentiable in the first variable due to \ref{thm:results:rsc:regularity}, we can apply Taylor's theorem, which states that for every $\x \in \R^n$ there exists $\lambda(\x) \in \intvcl{0}{1}$ such that
\begin{align}
\losstaylor{\x}{\grtrmu}{\yadv} &= \tfrac{1}{2} (\x - \grtrmu)^\T \gradient^2 \lossemp[\yadv](\grtrmu + \lambda(\x)(\x - \grtrmu))(\x - \grtrmu) \\
&= \tfrac{1}{2m} \sum_{i = 1}^m \loss''(\sp{\a_i}{\grtrmu + \lambda(\x)(\x - \grtrmu)}, \tilde{y}_i) \abs{\sp{\a_i}{\x - \grtrmu}}^2, \label{eq:proof:rsc:hessian}
\end{align}
where $\gradient^2 \lossemp[\yadv]$ denotes the Hessian matrix of $\lossemp[\yadv]$. If $\loss''$ would be bounded by a positive constant from below, we could proceed as in \cite[Lem.~4.4]{plan2015lasso}. But we merely assume that such a bound only holds in a certain neighborhood around the origin, so that the situation becomes more difficult. 
Therefore, let us first apply the assumption \ref{thm:results:rsc:strongconvexity} to find a simple lower bound for \eqref{eq:proof:rsc:hessian}.
For this purpose, we introduce the abbreviation $\h := \x - \grtrmu$ and make use of the standard inequality $(\sum_{i = 1}^m v_i^2)^{\frac{1}{2}} \geq \frac{1}{\sqrt{m}}\sum_{i = 1}^m \abs{v_i}$ for $v_1, \dots, v_m \in \R$:
\begin{align}
\losstaylor{\x}{\grtrmu}{\yadv}^{\frac{1}{2}} &\geq \tfrac{1}{\sqrt{2}m} \sum_{i = 1}^m \underbrace{\loss''(\sp{\a_i}{\grtrmu + \lambda(\x)\h}, \tilde{y}_i)^{\frac{1}{2}}}_{\substack{\\[3pt] \mathclap{\geq \sqrt{C_{M,N}}\cdot\indset{\intvop{-M}{M}\times\intvop{-N}{N}}(\sp{\a_i}{\grtrmu + \lambda(\x)\h}, \tilde{y}_i)}}} \abs{\sp{\a_i}{\h}} \\
&\geq \tfrac{\sqrt{C_{M,N}}}{\sqrt{2}m} \sum_{i = 1}^m \indset{M}(\sp{\a_i}{\grtrmu + \lambda(\x)\h}) \cdot \indset{N}(\tilde{y}_i) \cdot \abs{\sp{\a_i}{\h}}, \label{eq:proof:rsc:taylorloss}
\end{align}
where $\indset{M'} := \indset{\intvop{-M'}{M'}}$ for $M' > 0$.
Observing that
\begin{align}
\abs{\sp{\a_i}{\grtrmu + \lambda(\x)\h}} &\leq \abs{\sp{\a_i}{\grtrmu}} + \abs{\sp{\a_i}{\h}},\\
\abs{\tilde{y}_i} &\leq \abs{y_i} + \underbrace{\lnorm{\y - \yadv}[\infty]}_{=: \advdev_\infty} = \abs{\fobs(\sp{\a_i}{\grtr})} + \advdev_\infty, 
\end{align}
we can estimate the step functions in \eqref{eq:proof:rsc:taylorloss} by
\begin{align}
	\indset{M}(\sp{\a_i}{\grtrmu + \lambda(\x)\h}) &\geq \indset{\frac{M}{2}}(\sp{\a_i}{\grtrmu})  \indset{\frac{M}{2}}(\sp{\a_i}{\h}), \\
	\indset{N}(\tilde{y}_i) &\geq \indset{N}(\abs{\fobs(\sp{\a_i}{\grtr})}+ \advdev_\infty).
\end{align}
Moreover, we shall replace $\indset{\frac{M}{2}}(\sp{\a_i}{\h})$ by a ``hat indicator'' function:\footnote{Such a technical step is necessary because we will require Lipschitz continuity of this function later on.}
\begin{equation}
\indset{\frac{M}{2}}(\sp{\a_i}{\h}) \geq \indset{\frac{M}{2}}^{(s)}(\sp{\a_i}{\h}) := 
	\begin{cases} (\tfrac{M}{2} - \abs{\sp{\a_i}{\h}}) / \tfrac{M}{2}, & \abs{\sp{\a_i}{\h}} \leq \tfrac{M}{2}, \\
		0, & \text{otherwise}.
	\end{cases}
\end{equation}
This allows us to control $\losstaylor{\x}{\grtrmu}{\yadv}^{\frac{1}{2}}$ as follows:
\begin{align}
\losstaylor{\x}{\grtrmu}{\yadv}^{\frac{1}{2}} &\geq \tfrac{\sqrt{C_{M,N}}}{\sqrt{2}m} \sum_{i = 1}^m \abs{\underbrace{\indset{\frac{M}{2}}(\sp{\a_i}{\grtrmu}) \indset{N}(\abs{\fobs(\sp{\a_i}{\grtr})} + \advdev_\infty) \indset{\frac{M}{2}}^{(s)}(\sp{\a_i}{\h}) \sp{\a_i}{\h}}_{=: \empproc(\a_i,\h)}} \\
&= \tfrac{\sqrt{C_{M,N}}}{\sqrt{2}m} \sum_{i = 1}^m \abs{\empproc(\a_i,\h)}. \label{eq:proof:rsc:empprocess}
\end{align}
Hence, it suffices to (uniformly) bound the empirical process $\tfrac{\sqrt{C_{M,N}}}{\sqrt{2}m} \sum_{i = 1}^m \abs{\empproc(\a_i,\h)}$ by $C_\RSC^{1/2}\lnorm{\h}$. 
A very powerful and fairly general framework to prove uniform lower bounds for non-negative empirical processes was recently developed by Mendelson and collaborators \cite{koltchinskii2015smallball,lecue2014cs,mendelson2014diameter,mendelson2014learning,mendelson2014learninggeneral}, often referred to as \define{Mendelson's small ball method}.
Here, we shall adapt a variant of this approach which was presented by Tropp in \cite[Prop.~2.5.1]{tropp2014convex}.
Following his strategy, let us first introduce the following \define{marginal tail function} (also called \define{small ball function}):
\begin{equation}
Q_{v}(\sset_t) := \inf_{\h \in \sset_t} \prob[\abs{\empproc(\gaussian,\h)} \geq v], \quad v \geq 0, \quad \gaussian \distributed \Normdistr{\vec{0}}{\I},
\end{equation}
where we have put $\sset_t := (\sset - \grtrmu) \intersec t\S^{n-1}$.
The next proposition summarizes our version of Mendelson's small ball method:
\begin{proposition}[Mendelson's small ball method]\label{prop:proof:rsc:mendelson}
Under the assumptions of Theorem~\ref{thm:results:rsc}, we have:
\begin{thmlist}
\item\label{prop:proof:rsc:mendelson:process}
	The empirical process in \eqref{eq:proof:rsc:empprocess} has the following uniform lower bound on $\sset_t$:
	\begin{equation}\label{eq:proof:rsc:mendelson:process}
		\inf_{\h \in \sset_t} \tfrac{1}{m} \sum_{i = 1}^m \abs{\empproc(\a_i,\h)} \geq t \cdot Q_{2t}(\sset_t) - \frac{2\meanwidth{\sset_t}}{\sqrt{m}} - \frac{t \cdot u}{\sqrt{m}}
	\end{equation}
	with probability at least $1-\exp(-u^2/2)$ for any $u > 0$.
\item\label{prop:proof:rsc:mendelson:tail}
	The marginal tail function is bounded from below by a numerical constant $C_0$:
	\begin{equation}\label{eq:proof:rsc:mendelson:marginalbound}
		Q_{2t}(\sset_t) \geq C_0 > 0.
	\end{equation}
\end{thmlist}
\end{proposition}

Before proving these statements, let us finish the proof of Theorem~\ref{thm:results:rsc}, which is now a simple consequence of Proposition~\ref{prop:proof:rsc:mendelson}:
\begin{proof}[Proof of Theorem~\ref{thm:results:rsc}]
We apply Proposition~\ref{prop:proof:rsc:mendelson} with $u = \sqrt{2 C_4 m}$ ($C_4 > 0$ is a numerical constant which is chosen later) and assume that the event of \eqref{eq:proof:rsc:mendelson:process} has indeed occurred (with probability at least $1 - \exp(-C_4\cdot m)$). Then, for any $\x \in K \intersec (t \S^{n-1} + \grtrmu)$, we can control \eqref{eq:proof:rsc:empprocess} as follows:
\begin{align}
	\losstaylor{\x}{\grtrmu}{\yadv}^{\frac{1}{2}} &\geq \tfrac{\sqrt{C_{M,N}}}{\sqrt{2}m} \sum_{i = 1}^m \abs{\empproc(\a_i,\h)} 
		\geq \sqrt{\tfrac{C_{M,N}}{2}} \cdot \Big( t \cdot Q_{2t}(\sset_t) - \frac{2\meanwidth{\sset_t}}{\sqrt{m}} - \frac{t \cdot \sqrt{2 C_4 m}}{\sqrt{m}}\Big) \\
	&\stackrel{\eqref{eq:proof:rsc:mendelson:marginalbound}}{\geq} t \cdot \sqrt{\tfrac{C_{M,N}}{2}} \cdot \Big( C_0 - \frac{2\meanwidth{\sset_t}}{t \cdot \sqrt{m}} - \sqrt{2 C_4 } \Big).
\end{align}
Moreover, the assumption \eqref{eq:results:rsc:meanwidth} implies that
\begin{equation}
\frac{\meanwidth{\sset_t}}{t \cdot \sqrt{m}} = \frac{\meanwidth{(\sset - \grtrmu) \intersec t\S^{n-1}}}{t \cdot \sqrt{m}} \leq \frac{\meanwidth[t]{\sset - \grtrmu}}{t \cdot \sqrt{m}} \leq \sqrt{\frac{1}{C_3}}.
\end{equation}
Adjusting the numerical constants $C_3$ and $C_4$ appropriately (by increasing the number of observations and decreasing the probability of success, respectively), we finally obtain
\begin{equation}
\losstaylor{\x}{\grtrmu}{\yadv} \geq t^2 \cdot \underbrace{\tfrac{C_{M,N}}{2} \cdot \Big( C_0 - \frac{2}{\sqrt{C_3}} - \sqrt{2 C_4 } \Big)^2}_{=: C_\RSC > 0} \geq C_\RSC \cdot t^2 = C_\RSC  \lnorm{\x - \grtrmu}^2,
\end{equation}
which completes the proof.
\end{proof}

Now, we provide a proof of Proposition~\ref{prop:proof:rsc:mendelson}. This is essentially a repetition of the arguments from \cite[Prop.~2.5.1]{tropp2014convex} but in a slightly more general setting, where the scalar products $\sp{\a_i}{\h}$ are replaced by $\empproc(\a_i,\h)$.
\begin{proof}[Proof of Proposition~\ref{prop:proof:rsc:mendelson}]
\begin{proofsteps}
\item
	At first, let us apply Markov's inequality:
	\begin{equation}\label{eq:proof:rsc:mendelson:markov}
		\tfrac{1}{m} \sum_{i = 1}^m \abs{\empproc(\a_i,\h)} \geq \tfrac{t}{m} \sum_{i = 1}^m \indprob{\abs{\empproc(\a_i,\h)} \geq t}.
	\end{equation}
	Now, we introduce a directional version of the marginal tail function $Q_{2t}(\sset_t)$, namely
	\begin{equation}
		Q_{2t}(\h) := \prob[\abs{\empproc(\gaussian,\h)} \geq 2t], \quad \h \in \sset_t, \quad \gaussian \distributed\Normdistr{\vec{0}}{\I}.
	\end{equation}
	Obviously, $Q_{2t} (\sset_t) = \inf_{h \in \sset_t}Q_{2t}(\h)$. Adding and subtracting this quantity in \eqref{eq:proof:rsc:mendelson:markov} and taking the infimum leads us to
	\begin{align}
		\inf_{\h \in \sset_t}\tfrac{1}{m} \sum_{i = 1}^m \abs{\empproc(\a_i,\h)} &\geq \inf_{\h \in \sset_t} \Big( t\cdot Q_{2t}(\h) - \tfrac{t}{m} \sum_{i = 1}^m (Q_{2t}(\h) - \indprob{\abs{\empproc(\a_i,\h)} \geq t}) \Big) \\
		&\geq t \cdot Q_{2t} (\sset_t) - \tfrac{t}{m} \cdot \sup_{\h \in \sset_t} \sum_{i = 1}^m (Q_{2t}(\h) - \indprob{\abs{\empproc(\a_i,\h)} \geq t}). \label{eq:proof:rsc:mendelson:mendelson}
	\end{align}
	Similarly to \cite[Prop.~2.5.1]{tropp2014convex}, we can now apply a deviation inequality for bounded differences in order to control this supremum (see also \cite[Thm.~6.2]{boucheron2013concentration}):
	\begin{equation}
	\sup_{\h \in \sset_t} \sum_{i = 1}^m (Q_{2t}(\h) - \indprob{\abs{\empproc(\a_i,\h)} \geq t}) \leq \mean[\sup_{\h \in \sset_t} \sum_{i = 1}^m (Q_{2t}(\h) - \indprob{\abs{\empproc(\a_i,\h)} \geq t})] + u \sqrt{m} \label{eq:proof:rsc:mendelson:deviation}
	\end{equation}
	with probability at least $1 - \exp(-u^2/2)$. Next, let us estimate the expected supremum by an empirical \define{Rademacher process}. For this, consider the soft indicator function
	\begin{equation}
		\psi_t \colon \R \to \intvcl{0}{1}, \ v \mapsto \psi_t(v) := \begin{cases}
			0, & \abs{v} \leq t, \\
			(\abs{v} - t) / t, & t < \abs{v} \leq 2t, \\
			1, & 2t < \abs{v},
		\end{cases}
	\end{equation}
	satisfying $\indprob{\abs{v} \geq 2t} \leq \psi_t(v) \leq \indprob{\abs{v} \geq t}$ for all $v \in \R$. By a classical symmetrization argument (cf. \cite[Lem.~6.3]{ledoux1991pbs}), we obtain
	\begin{align}
		&\mean[\sup_{\h \in \sset_t} \sum_{i = 1}^m (Q_{2t}(\h) - \indprob{\abs{\empproc(\a_i,\h)} \geq t})] \\
		={}& \mean[\sup_{\h \in \sset_t} \sum_{i = 1}^m (\mean[\indprob{\abs{\empproc(\a_i,\h)} \geq 2t}] - \indprob{\abs{\empproc(\a_i,\h)} \geq t})] \\
		\leq{}& \mean[\sup_{\h \in \sset_t} \sum_{i = 1}^m (\mean[\psi_t(\empproc(\a_i,\h))] - \psi_t(\empproc(\a_i,\h)))] \\
		\leq{}& 2 \mean[\sup_{\h \in \sset_t} \sum_{i = 1}^m \eps_i \psi_t(\empproc(\a_i,\h))] = \frac{2}{t} \cdot \mean_{\A}\mean_{\eps}[\sup_{\h \in \sset_t} \sum_{i = 1}^m \eps_i \cdot  t\psi_t(\empproc(\a_i,\h))],\label{eq:proof:rsc:mendelson:symmetrization}
	\end{align}
	where $\eps_1, \dots, \eps_m$ are independent Rademacher variables.\footnote{That is, $\prob[\eps_i = 1] = \prob[\eps_i = -1] = \frac{1}{2}$.} Note that in the last step, we have used Fubini's theorem to separate the expectations over $\A$ and $\eps_1, \dots, \eps_m$. Now, we would like to invoke a contraction principle to bound the inner expectation (over the Rademacher variables). For this, let us observe that $\empproc(\a_i,\h)$ depends on $\h$ only in terms of $\sp{\a_i}{\h}$. Thus, we may define $\empproc_i(\sp{\a_i}{\h}) := t\psi_t(\empproc(\a_i,\h))$, where $\empproc_i$ still depends on $\a_i$ but not on $\h$ anymore. Recalling the definition of $\empproc$ and $\psi_t$, one easily checks that $\empproc_i$ is in fact a contraction, i.e.,
	\begin{equation}
		\abs{\empproc_i(v_1) - \empproc_i(v_2)} \leq \abs{v_1 - v_2}, \quad \text{for all $v_1, v_2 \in \R$,}
	\end{equation}
	Since $\empproc_i(0) = 0$, we can apply the Rademacher comparison principle from \cite[Eq.~(4.20)]{ledoux1991pbs}, where we perform a change of variables with $\vec{v} = (v_1, \dots, v_n) := \A \h$:
	\begin{align}
		\mean_{\eps}[\sup_{\h \in \sset_t} \sum_{i = 1}^m \eps_i \cdot  t\psi_t(\empproc(\a_i,\h))] &= \mean_{\eps}[\sup_{\h \in \sset_t} \sum_{i = 1}^m \eps_i \empproc_i(\sp{\a_i}{\h})] = \mean_{\eps}[\sup_{\vec{v} \in \A\sset_t} \sum_{i = 1}^m \eps_i \empproc_i(v_i)] \\
		&\leq \mean_{\eps}[\sup_{\vec{v} \in \A\sset_t} \sum_{i = 1}^m \eps_i v_i] = \mean_{\eps}[\sup_{\h \in \sset_t} \sum_{i = 1}^m \eps_i \sp{\a_i}{\h}].
	\end{align}
	Applying this to \eqref{eq:proof:rsc:mendelson:symmetrization}, we finally obtain
	\begin{align}
		&\mean[\sup_{\h \in \sset_t} \sum_{i = 1}^m (Q_{2t}(\h) - \indprob{\abs{\empproc(\a_i,\h)} \geq t})] \leq \frac{2}{t} \cdot \mean_{\A}\mean_{\eps}[\sup_{\h \in \sset_t} \sum_{i = 1}^m \eps_i \sp{\a_i}{\h}] \\
		={}& \frac{2}{t} \cdot \mean[\sup_{\h \in \sset_t} \sum_{i = 1}^m \sp{\a_i}{\h}] = \frac{2\sqrt{m}}{t} \cdot \mean[\sup_{\h \in \sset_t} \sp{\gaussian}{\h}] = \frac{2\sqrt{m}}{t} \cdot \meanwidth{\sset_t},
	\end{align}
	where we have used that $\sum_{i = 1}^m \eps_i\a_i \distributed \sum_{i = 1}^m \a_i =: \sqrt{m} \cdot \gaussian \distributed \Normdistr{\vec{0}}{m\I}$. This estimate, combined  with \eqref{eq:proof:rsc:mendelson:deviation}, yields the desired lower bound for \eqref{eq:proof:rsc:mendelson:mendelson}:
	\begin{align}
		\inf_{\h \in \sset_t}\tfrac{1}{m} \sum_{i = 1}^m \abs{\empproc(\a_i,\h)} &\geq t \cdot Q_{2t} (\sset_t) - \frac{t}{m} \cdot \Big(\frac{2\sqrt{m}}{t} \cdot \meanwidth{\sset_t} + u \sqrt{m}\Big) \\
		&= t \cdot Q_{2t}(\sset_t) - \frac{2\meanwidth{\sset_t}}{\sqrt{m}} - \frac{t \cdot u}{\sqrt{m}}.
	\end{align}
\item
	Let us fix some $\h \in \sset_t$. Noting that $\indset{M/2}^{(s)}(\sp{\a_i}{\h}) \geq \tfrac{1}{2} \indset{M/4}(\sp{\a_i}{\h})$, we can estimate the directional marginal tail function as follows (for any $i \in [m]$):
	\begin{align}
		Q_{2t}(\h) &= \prob[\abs{\empproc(\gaussian,\h)} \geq 2t] \\
		&= \prob[\abs{\indset{\frac{M}{2}}(\sp{\a_i}{\grtrmu}) \indset{N}(\abs{\fobs(\sp{\a_i}{\grtr})}+ \advdev_\infty) \indset{\frac{M}{2}}^{(s)}(\sp{\a_i}{\h}) \sp{\a_i}{\h}} \geq 2t] \\
		&\geq \prob[\abs{\indset{\frac{M}{2}}(\sp{\a_i}{\grtrmu}) \indset{N}(\abs{\fobs(\sp{\a_i}{\grtr})}+ \advdev_\infty) \cdot \tfrac{1}{2} \indset{\frac{M}{4}}(\sp{\a_i}{\h}) \sp{\a_i}{\h}} \geq 2t] \\
		&= \prob[\abs{\sp{\a_i}{\grtrmu}} \leq \tfrac{M}{2},\ \abs{\fobs(\sp{\a_i}{\grtr})}+ \advdev_\infty \leq N,\ 4t \leq \abs{\sp{\a_i}{\h}} \leq \tfrac{M}{4}] \\
		&\geq \prob[\abs{\sp{\a_i}{\grtrmu}} \leq \tfrac{M}{2},\ \abs{\fobs(\sp{\a_i}{\grtr})}+ \advdev_\infty \leq N,\ 4 \leq \abs{\sp{\a_i}{\tfrac{\h}{t}}} \leq 8], \label{eq:proof:rsc:mendelson:marginal}
	\end{align}
	where in the last step, we have divided by $t$ and used that $32t \leq M$. Note that $\gaussianuniv_0 := \sp{\a_i}{\grtr}$ and $\gaussianuniv := \sp{\a_i}{\tfrac{\h}{t}}$ are standard Gaussians whose probability distribution particularly does not dependent on $t$. To bound \eqref{eq:proof:rsc:mendelson:marginal} from below, we first estimate the complementary events by Markov's inequality:
	\begin{align}
		\prob[\abs{\sp{\a_i}{\grtrmu}} > \tfrac{M}{2}] &\leq \frac{2\abs{\scalfac}\mean[\abs{\gaussianuniv_0}]}{M} = \frac{2\abs{\scalfac}\sqrt{\frac{2}{\pi}}}{M}, \\
		\prob[\abs{\fobs(\sp{\a_i}{\grtr})} > N - \advdev_\infty] &\leq 
		\frac{\mean[\abs{\fobs(\gaussianuniv_0)}]}{N - \advdev_\infty}.\label{eq:proof:rsc:mendelson:complementary}
	\end{align}
	Putting $C' := \prob[\abs{\gaussianuniv} \in \intvclop{0}{4} \union \intvop{8}{\infty}] < 1$, we end up with
	\begin{align}
		Q_{2t}(\h) &\geq \prob[\abs{\sp{\a_i}{\grtrmu}} \leq \tfrac{M}{2},\ \abs{\fobs(\sp{\a_i}{\grtr})} + \advdev_\infty \leq N,\ 4 \leq \abs{\sp{\a_i}{\tfrac{\h}{t}}} \leq 8] \\
		&\geq 1 - \prob[\abs{\sp{\a_i}{\grtrmu}} > \tfrac{M}{2}] - \prob[\abs{\fobs(\sp{\a_i}{\grtr})} > N - \advdev_\infty] - C' \\
		&\geq 1 - \frac{2\abs{\scalfac}\sqrt{\frac{2}{\pi}}}{M} - \frac{\mean[\abs{\fobs(\gaussianuniv_0)}]}{N - \advdev_\infty} - C' =: C_0.
	\end{align}
	Due to the assumptions on $M$ and $N$ in \ref{thm:results:rsc:strongconvexity}, with $C_1$ and $C_2$ appropriately chosen, it follows that $C_0$ is indeed positive. Since $C_0$ is also independent of $\h$, we can take the infimum over all $\h \in \sset_t$, which proves the claim. \qedhere
\end{proofsteps}
\end{proof}

\begin{remark}
\begin{rmklist}
\item
	The \emph{small ball function} $Q_{2t}(\sset_t)$ captures the probability that $\empproc(\gaussian,\h)$ is not too close to zero for any vector $\h \in \sset - \grtrmu$ lying on a ``small'' sphere $tS^{n-1}$. Thus, if this quantity is bounded from below, the corresponding empirical process \eqref{eq:proof:rsc:empprocess} is very likely to be distant from zero as well (as long as $m$ is sufficiently large; cf. Proposition~\ref{prop:proof:rsc:mendelson}\ref{prop:proof:rsc:mendelson:process}).
	In particular, it has turned out that the explicit definition of $\empproc(\gaussian,\h)$ is only involved the second part of Proposition~\ref{prop:proof:rsc:mendelson}, where we have dealt with $Q_{2t}(\sset_t)$.
\item
	The (numerical) constants in the previous proofs may be improved. For instance, $C'$ could be chosen smaller in the second part of Proposition~\ref{prop:proof:rsc:mendelson} if the ``jump'' of soft indicator function $\psi_t$ in the first part is made ``sharper.'' Moreover, one might invoke tighter bounds than Markov's inequality in \eqref{eq:proof:rsc:mendelson:complementary}. But since this would not concern the qualitative behavior of the approximation error, we leave these details to the reader.
	Apart from that, the above argumentation does not rely on the definition of $\scalfac$ in \eqref{eq:results:recovery:modelparam:scalfac}, so that Theorem~\ref{thm:results:rsc} actually holds for any $\scalfac \neq 0$.
\item
	Using a \define{chaining argument} similarly to \cite[Thm.~2.6.3]{tropp2014convex}, one might even extend the statement of Theorem~\ref{thm:results:rsc} to sub-Gaussian vectors. But note that the underlying random distribution must not be too ``spiky,'' since otherwise, the small ball function $Q_{2t}(\sset_t)$ cannot be appropriately controlled anymore. \qeddiamond
\end{rmklist}
\end{remark}


\section*{Acknowledgements}
The author thanks Roman Vershynin and Gitta Kutyniok for fruitful discussions.
This research is supported by the Einstein Center for Mathematics Berlin (ECMath), project grant CH2.

\renewcommand*{\bibfont}{\small}
\bibliographystyle{abbrv}
\bibliography{general-estimation}

\end{document}